\documentclass{amsart}
\usepackage{amsmath,mathtools,amssymb,amsfonts,parskip,tikz-cd,tikz,hyperref,enumitem,amsthm,extarrows,comment,soul,xcolor,scalerel,thmtools}
\hypersetup{
	colorlinks = true,
	linkcolor = teal
	}
\usetikzlibrary{decorations.pathmorphing}
\usetikzlibrary{decorations.markings}
\usepackage[nameinlink]{cleveref}
\usepackage[margin = 1in]{geometry}

\usepackage[sc,osf]{mathpazo}
\usepackage[mathscr]{euscript}

\tikzcdset{sstyle/.style={nodes = {inner sep=2pt}, column sep = small, row sep = small}} 
\tikzcdset{tstyle/.style={nodes = {inner sep=2pt}, column sep = tiny, row sep = tiny}} 
\tikzset{Rightarrow/.style={double equal sign distance,>={Implies},->}, triple/.style={-,preaction={draw,Rightarrow}}}

\renewcommand{\tilde}[1]{\widetilde{#1}}

\newcommand{\undl}[1]{\underline{#1}}
\newcommand{\op}{\mathrm{op}}
\newcommand{\orient}{\mathrm{or}}
\newcommand{\el}{\mathrm{el}}

\newcommand{\sA}{\mathscr{A}}

\newcommand{\sC}{\mathscr{C}}
\newcommand{\sD}{\mathscr{D}}
\newcommand{\sE}{\mathscr{E}}

\newcommand{\EE}{\mathbb{E}}
\newcommand{\RR}{\mathbb{R}}

\newcommand{\LL}{\mathbb{L}}

\newcommand{\fCat}{\mathfrak{Cat}}

\newcommand{\Spaces}{\mathfrak{S}}
\newcommand{\bD}{\mathbf{\Delta}}
\newcommand{\bN}{\mathbf{N}}
\newcommand{\Tang}{\mathrm{Tang}}
\newcommand{\Bord}{\mathrm{Bord}}
\newcommand{\Adj}{\mathrm{Adj}}

\newcommand{\Map}{\mathrm{Map}}
\newcommand{\Fun}{\mathrm{Fun}}

\newcommand{\Cat}{\mathrm{Cat}}
\newcommand{\Seg}{\mathrm{Seg}}
\newcommand{\Glb}{\mathrm{Glb}}
\newcommand{\GlbSeg}{\mathrm{GlbSeg}}

\newcommand{\Sq}{\mathrm{Sq}}
\newcommand{\id}{\mathrm{id}}
\newcommand{\ev}{\mathrm{ev}}
\newcommand{\zig}{\mathcal{Z}}

\DeclareMathOperator*{\colim}{colim}

\newcommand{\sub}{\subseteq}

\newcommand{\vphi}{\varphi}

\theoremstyle{definition}
\newtheorem{dfn}{Definition}[section]
\newtheorem{rmk}[dfn]{Remark}

\newtheorem{exm}[dfn]{Example}

\theoremstyle{theorem}
\newtheorem{prp}[dfn]{Proposition}

\newtheorem{lmm}[dfn]{Lemma}
\newtheorem{thm}[dfn]{Theorem}
\newtheorem{crl}[dfn]{Corollary}
\newtheorem{cnj}[dfn]{Conjecture}
\newtheorem{fct}[dfn]{Fact}

\newcounter{mainthms}

\theoremstyle{theorem}
\newtheorem{mainthm}[mainthms]{Theorem}
\newtheorem{maincnj}[mainthms]{Conjecture}

\title{Zigzags and free adjunctions}

\author{Lorenzo Riva}
\address{Center for Mathematical Sciences and Applications, Harvard University, Cambridge, USA}
\email{lorenzo@cmsa.fas.harvard.edu}

\author{Martina Rovelli}
\address{Department of Mathematics and Statistics, University of Massachusetts Amherst, Amherst, USA}
\email{mrovelli@umass.edu}
\address{Department of Mathematics and Statistics, University of Ottawa, Ottawa, Canada}
\email{mrovelli@uottawa.ca}

\date{\today}

\begin{document}

\maketitle

\begin{abstract}
	We construct an explicit combinatorial model of the functor which adds right adjoints to the morphisms of an $\infty$-category, and we speculate on possible extensions to higher dimensions.
\end{abstract}

\setcounter{tocdepth}{1}
\tableofcontents

\section{Introduction} \label{sec:intro}

\subsection{Motivation: adding adjoints}

Fix an $(\infty,n)$-category $\sC$ and two $k$-morphisms $f : x \to y$ and $g : y \to x$ in $\sC$ between parallel $(k-1)$-morphisms $x$ and $y$. Then $f$ is a \emph{left adjoint} of $g$ (and $g$ is a \emph{right adjoint} of $f$) if there are $(k+1)$-morphisms $\eta : \id(x) \to g \circ_k f$ and $\varepsilon : f \circ_k g \to \id(y)$ in $\sC$ satisfying the snake equations
\begin{equation*}
	(\varepsilon \circ_k \id(f)) \circ_{k+1} (\id(f) \circ_k \eta) \simeq \id(f), \qquad (\id(g) \circ_k \varepsilon) \circ_{k+1} (\eta \circ_k \id(g)) \simeq \id(g).
\end{equation*}
If, in addition, $\sC$ is symmetric monoidal, then we can similarly define a notion of \emph{left/right adjoints} (more commonly called \emph{duals}) for objects. 

Given any homotopy commutative monoid $S$ (i.e. $S$ is an $\EE_\infty$-algebra in the $\infty$-category $\Spaces$ of spaces) we can ask the following question: what is the smallest symmetric monoidal $(\infty,n)$-category $F^{n, \mathrm{adj}}(S)$ such that (1) $S$ is contained in the space of objects of $F^{n, \mathrm{adj}}(S)$ and (2) for $0 \leq k < n$, all $k$-morphisms of $F^{n, \mathrm{adj}}(S)$ have a left and right adjoint? Notice that this is a purely categorical property, so it is very surprising that the answer, at least in a special case, appears to be purely geometric: the \emph{cobordism hypothesis} \cite{BD1995,Lurie2009} asserts that, when $S = A^1$ is generated by a single element, $F^{n, \mathrm{adj}}(A^1)$ is equivalent to the $(\infty,n)$-category $\mathrm{Bord}_n^{\mathrm{fr}}$ whose $k$-morphisms are $n$-framed compact $k$-manifolds with corners. A reference for this statement can be found in \cite{Lurie2009} and the arguments depend crucially on decomposing of manifolds into suitable handles, each of which witnesses an adjunction.

We can imagine a different strategy: first build $F^{n, \mathrm{adj}}(S)$ for any $S$ with purely categorical or combinatorial tools, and then exhibit an explicit equivalence $F^{n, \mathrm{adj}}(A^1) \simeq \mathrm{Bord}_n^{\mathrm{fr}}$, thus bypassing the need to show that $\mathrm{Bord}_n^{\mathrm{fr}}$ has the desired universal property. Variations of this strategy have been used in dimensions $n \leq 3$ for this and similar problems -- see, for example, \cite{Abrams1996,SP2009,Juhasz2018} -- but they rely on our understanding of low-dimensional manifolds. Nevertheless, it is interesting to think whether we can exploit some \emph{partial} knowledge of high-dimensional manifolds, like their handle decompositions, to make this strategy work. The question then turns to: how do we build $F^{n, \mathrm{adj}}(S)$?

In this paper we answer this question in the slightly different context of freely adding right adjoints to the morphisms of an $(\infty,1)$-category $\sC$. We can describe an $(\infty,2)$-category $F^{2, \mathrm{ladj}}(\sC)$
informally as follows:
\begin{enumerate}[start = 0]
	\item its objects are the objects of $\sC$;
	\item its $1$-morphisms are formal zigzags of $1$-morphisms of $\sC$, with composition given by concatenation;
	\item its $2$-morphisms are generated under vertical composition by formal zigzags of commutative squares in $\sC$, subject to the condition that the first and last vertical legs of each zigzags are invertible, with a minimal set of relations.
\end{enumerate}
For example, given any $1$-morphism $f$ of $\sC$, pictured as an arrow $f = (x \xlongrightarrow{f} y)$, there should be a corresponding $1$-morphism $f'$ going in the opposite direction of $f$, pictured as $f' = (y \xlongleftarrow{f} x)$, and since we are enforcing that $F^{2, \mathrm{ladj}}(\sC)$ be an $(\infty,2)$-category we must also have composites of the form
\begin{equation*}
	g' \circ_1 f = (x \xlongrightarrow{f} y \xlongleftarrow{g} x), \qquad h \circ_1 g' \circ_1 f = (x \xlongrightarrow{f} y \xlongleftarrow{g} x \xlongrightarrow{h} y),
\end{equation*}
and so on. The generating $2$-morphisms of $F^{n, \mathrm{ladj}}(\sC)$ are of the form
\begin{equation*}
	\begin{tikzcd}
		x_1 \ar[r, "f_1"] \ar[d, "\simeq"'] & x_2 \ar[d, "g_1"] & x_3 \ar[l, "f_2"'] \ar[d, "g_2"] \ar[r, "f_3"] & x_4 \ar[d, "g_3"] & x_5 \ar[l, "f_4"'] \ar[d, "\simeq"] \\
		y_1 \ar[r, "h_1"'] & y_2 & y_3 \ar[l, "h_2"] \ar[r, "h_3"'] & y_4 & y_5 \ar[l, "h_4"] 
	\end{tikzcd}
\end{equation*}
where each individual square is a commutative square in $\sC$ and the left and right vertical morphisms are invertible. In particular, given each $f$ in $\sC$ we have two $2$-morphisms
\begin{equation*}
	\eta := \,
	\begin{tikzcd}
		x \ar[r, equal] \ar[d, equal] & x \ar[d, "f"] & x \ar[l, equal] \ar[d, equal] \\
		x \ar[r, "f"'] & y & x \ar[l, "f"]
	\end{tikzcd}
	\qquad
	\varepsilon_f := \,
	\begin{tikzcd}
		y \ar[d, equal] & x \ar[l, "f"'] \ar[d, "f"] \ar[r, "f"] & y \ar[d, equal] \\
		y \ar[r, equal] & y & y \ar[l, equal]
	\end{tikzcd}
\end{equation*}
which exhibit $f$ as a left adjoint of $f'$: one of the snake equations is depicted as
\begin{equation*}
	\begin{tikzcd}
		x \ar[r, equal] \ar[d, equal] & x \ar[d, "f"] \ar[r, equal] & x \ar[d, equal] \ar[r, "f"] & y \ar[d, equal] \\
		x \ar[r, "f"] \ar[d, equal] & y \ar[d, equal] & x \ar[l, "f"'] \ar[d, "f"] \ar[r, "f"] & y \ar[d, equal] \\
		x \ar[r, "f"] & y \ar[r, equal] & y \ar[r, equal] & y
	\end{tikzcd}
	\, \simeq \,
	\begin{tikzcd}
		x \ar[r, equal] \ar[d, equal] & x \ar[d, "f"] & x \ar[l, equal] \ar[d, "f"] \ar[r, "f"] & y \ar[d, equal] \\
		x \ar[r, "f"'] & y & y \ar[l, equal] \ar[r, equal] & y
	\end{tikzcd}
	\, \simeq \,
	\begin{tikzcd}
		x \ar[r, "f"] \ar[d, equal] & y \ar[d, equal] \\
		x \ar[r, "f"'] & y
	\end{tikzcd}
\end{equation*}
where the first equivalence is obtained by composing the diagram vertically and the second by composing it horizontally (we can do this since $2$-morphisms in an $(\infty,2)$-category satisfy the \emph{interchange law}), and the other snake equation is similar. Constructions of this type have appeared in the literature to answer similar questions about adding adjoints -- see, for example, Dawson-Par\'e-Pronk's work on strict $2$-categories of fences \cite{DPP2003} or Cnossen-Lez-Linskens' work on span $2$-categories \cite{CLL2025} -- or inverses -- see Dwyer-Kan's hammock localization of simplicial categories \cite{DK1980} or the general calculus of fractions \cite{GZ1967}. Moreover, using work of Loubaton-Ruit \cite{LR2025} on the universal property of the double $\infty$-category of squares in $\sC$, it's not surprising to find that the $(\infty,2)$-category that we described above also has some kind of universal property. 

\subsection{Main results}

Describing an $(\infty,2)$-category using its $k$-morphisms hides a lot of data and does not make clear whether the construction accounts for higher coherence data. The first main result is that this informal description can be upgraded to a formal one: 
\begin{mainthm}[{\Cref{dfn:zigzag}, \Cref{lmm:vert-decomp}, \Cref{prp:horiz-decomp}}]
	There is a functor $\zig_{+}^2 : \fCat_{(\infty,1)} \to \fCat_{(\infty,2)}$ such that, for any $\sC \in \fCat_{(\infty,1)}$,
	\begin{enumerate}[start = 0]
		\item the space of objects of $\zig_{+}^2(\sC)$ is the space of objects of $\sC$;
		\item the space of $1$-morphisms of $\zig_{+}^2(\sC)$ is the space of formal zigzags of $1$-morphisms of $\sC$, with composition given by concatenation;
		\item the space of $2$-morphisms of $\zig_{+}^2(\sC)$ is generated under vertical composition by the space of formal zigzags of commutative squares in $\sC$, subject to the condition that the first and last vertical legs of each zigzags are invertible, with a minimal set of relations.
	\end{enumerate}
	Moreover, there is an inclusion $i : \sC \to \zig_{+}^2(\sC)$ (natural in $\sC$) which sends all $1$-morphisms of $\sC$ to $1$-morphisms with a right adjoint in $\zig_{+}^2(\sC)$.
\end{mainthm}

Secondly, we show that $\zig_{+}^2(\sC)$ is universal with respect to that property:
\begin{mainthm}[{\Cref{crl:univ-prop}, \Cref{crl:adj}}]
	Consider the functor $\tau_1 (-)^{\mathrm{ladj}} : \fCat_2 \to \fCat_1$ which computes the sub-$\infty$-category containing only the $1$-morphisms which are left adjoints. Then $\zig_{+}^2$ is a left adjoint of $\tau_1 (-)^{\mathrm{ladj}}$, with the inclusion $\sC \to \zig_{+}^2(\sC)$ as the unit: for any $(\infty,2)$-category $\sD$,
	\begin{equation*}
		\Map(\zig_{+}^2(\sC), \sD) \simeq \Map(\sC, \tau_1 \sD^{\mathrm{ladj}}).
	\end{equation*}
    In particular there is an equivalence $\zig^2([1]) \simeq \Adj$, where $[1]$ is the walking $1$-morphism and $\Adj$ is the walking adjunction.
\end{mainthm}
In the discussion following \Cref{crl:adj} we show why the above theorem also has a surprisingly geometric analogue using the well-known geometric formulation of the walking adjunction $\Adj$ from \cite{RV2016}. 

This zigzag construction can be extended to higher dimensions. Indeed, the functor $\zig_{+}^2$ is part of a family of functors $\zig_{+}^{n+1} : \fCat_{(\infty,n)} \to \fCat_{(\infty,n+1)}$ such that the $k$-morphisms of $\zig_{+}^{n+1}(\sC)$ can be similarly described as zigzags -- see \Cref{dfn:zigzag}. We show that each $\zig_{+}^{n+1}$ adds right adjoints:
\begin{mainthm}[{\Cref{prp:mainthm2}}]
	Fix $n \geq 1$. For any $\sC \in \fCat_{(\infty,n)}$ there is a map $\sC \to \zig_{+}^{n+1}(\sC)$ which, for $1 \leq k \leq n$, sends every $k$-morphism in $\sC$ to a $k$-morphism which has a right adjoint in $\zig_{+}^{n+1}(\sC)$.
\end{mainthm}

The problem of whether $\zig_{+}^{n+1}(\sC)$ is initial with respect to this property remains undecided, but we expect this partial result:
\begin{maincnj}[{\Cref{cnj:mainconj}}]
	Fix $n \geq 1$. For any $1 \leq k \leq n+1$, the space of $k$-morphisms of $\zig_{+}^{n+1}(\sC)$ is generated under composition by the $k$-morphisms of $\sC$ and the adjunction (co)units for lower-dimensional morphisms of $X$.
\end{maincnj}

\subsection{Future work}

Ultimately we hope to connect our zigzag construction with the cobordism hypothesis. Here we sketch our strategy and our hopes for future results. Consider $F^k$, the free $\EE_k$-algebra in spaces on a single (non-identity) generator. Using the stabilization hypothesis (see, for example, \cite[Proposition 10.11]{Haugseng2018}), $F^k$ induces an $(\infty,k)$-category $B^k F^k$ with a contractible space of $j$-morphisms for all $0 \leq j < k$. Then $\zig_{+}^{k+n}(B^k F^k)$ is an $(\infty,k+n)$- category with a contractible space of $j$-morphisms for all $0 \leq j < k$, and thus it induces an $\EE_k$-monoidal $(\infty,n)$-category $T^{k,n} := \Omega^k \zig_{+}^{k+n}(B^k F^k)$. Our results tell us that $T^{k,n}$ has left adjoints for all $j$-morphisms when $0 \leq j \leq n-1$, and \Cref{crl:ambidextrous} further ensures that those left adjoints are also right adjoints when $j \leq n-2$. This is a similar property to that of the \emph{oriented tangle category} $\Tang_{k,n}^{\orient}$, an $\EE_k$-monoidal $(\infty,n)$-category constructed in \cite{AF2017} whose $j$-morphisms are compact oriented $j$-manifolds with corners with an embedding in $\RR^{k+j}$.

\begin{maincnj}[{\Cref{cnj:secondcnj}}]
	For any $k, n, \geq 1$ there exists a map $T^{k,n} \to \Tang_{k,n}^{\orient}$ sending the generator $\ast \in F^k \sub T^{k,n}$ to the object $(\{0\} \hookrightarrow \RR^k) \in \Tang^\orient_{k,n}$.
\end{maincnj}

At the end of \Cref{sec:extensions} we give an idea of how this map can be constructed. Taking colimits as $k \to \infty$ we would obtain a map $\colim_k T^{k,n} \to \Bord_{n}^{\orient}$, the oriented cobordism category. We hope that we can use this map to get a better grasp on the latter $(\infty,n)$-category and on the oriented cobordism hypothesis -- in particular, we'd like to investigate how close this map is to being an equivalence.

\subsection{Organization of the paper and notation}

While our main results are framed in the language of $(\infty,n)$-categories, the technical work is done with higher Segal spaces; to get back the results one just has to start with a higher Segal space that is, in addition, complete. We start by recalling some categorical notions in \Cref{sec:segal}, mostly concerning the theory of higher Segal spaces. We prove a general formula that extracts a globular $n$-uple Segal space (equivalently, a $2$-fold Segal space) from any $n$-simplicial space, and we use that formula in \Cref{sec:squares} to construct the zigzagification functor $\zig_{+}^{n}$. In \Cref{sec:properties} we prove the universal property of $\zig_{+}^2$ by exploiting some results about generating spaces of $2$-morphisms. Finally, in \Cref{sec:extensions} we discuss the general case and talk about future work on the connection with the tangle hypothesis. 

We work mostly model independently. For each $n$ we fix an $\infty$-category $\fCat_{(\infty,n)}$ of $(\infty,n)$-categories and set $\Spaces := \fCat_{(\infty,0)}$, which we call the $\infty$-category of spaces. There is a truncation functor $\tau_m : \fCat_{(\infty,n)} \to \fCat_{(\infty,m)}$ for each $m \leq n$ which is a right adjoint to the canonical inclusion $\fCat_{(\infty,m)} \sub \fCat_{(\infty,n)}$. In any $(\infty,n)$-category we have composition operations $- \circ_k -$ for each $1 \leq k \leq n$ (where the subscript $k$ denotes that we are composing in the $k$th direction, i.e. along a shared $(k-1)$-morphism) and an identity operation $\id(-)$ which takes a $k$-morphism to its identity $(k+1)$-morphism. We use $\partial^0$ and $\partial^1$ to denote the source and target operations, respectively. We use $(-)^{\op_i} : \fCat_n \to \fCat_n$ to denote the functor switching the source and target of $i$-morphisms.

We denote by $\bD$ the (ordinary) category of simplices, whose objects $[t] \in \bD$ are the ordered sets ${0 < 1 < \dotsb < t}$ and whose morphisms are the order-preserving functions. This category comes with a canonical map $\bD \to \fCat_1$ and we tend to think of $[t]$ as its image under this map. An object of the product $\bD^n$ is denoted by $[\undl{t}]$, where $\undl{t} = (t_1, \dotsc, t_n)$ is an $n$-tuple of non-negative integers. Tuples of the form $(a, \dotsc, a)$ will be abbreviated as $\undl{a^n}$. We write $\bD^{n, \op}$ for $(\bD^{n})^{\op}$.

\subsection{Acknowledgements}

We kindly thank Chris Schommer-Pries, Stephan Stolz, Dan Freed, and Tomer Schlank for enlightening conversations about this work. The first author is generously supported by the Simons Foundation on Global Categorical Symmetries (SFI-MPS-GCS-00008528-09) and the Center for Mathematical Sciences and Applications.

\section{Categorical preliminaries} \label{sec:segal}

We start by recalling some of the theory of higher Segal spaces and their specializations all the way to complete globular $n$-uple Segal spaces. The only original results are in \Cref{sec:underlying} and serve a purely technical (though vital) role in the paper, so the readers who are familiar with higher Segal spaces should feel free to skip to the next section and refer back to the new results when needed.

\subsection{Higher Segal spaces}

Fix an $\infty$-category $\sA$ with finite limits. 

\begin{dfn}
	A \emph{Segal object in $\sA$} is a functor $X : \bD^\op \to \sA$ such that, for every $n$, the canonical projection map $X_n \to X_1 \times_{X_0} \dotsb \times_{X_0} X_1$ induced by the inert inclusions $[1] \hookrightarrow [n]$ is an equivalence; this is referred to as the \emph{Segal condition}. We denote by $\Seg(\sA)$ the full sub-$\infty$-category of $\Fun(\bD^\op, \sA)$ spanned by the Segal objects in $\sA$.
\end{dfn}

\begin{rmk}
	Limits in $\Seg(\sA)$ are computed pointwise. In particular, $\Seg(\sA)$ has all finite limits.
\end{rmk}

\begin{dfn}
	The $\infty$-category $\Seg^n(\sA)$ of \emph{$n$-uple Segal objects in $\sA$} is defined inductively as follows:
	\begin{itemize}
		\item $\Seg^0(\sA) := \sA$;
		\item $\Seg^1(\sA) := \Seg(\sA)$;
		\item $\Seg^{n}(\sA) := \Seg(\Seg^{n-1}(\sA))$ for $n \geq 2$.
	\end{itemize}
\end{dfn}

We will mostly concentrate on the case $\sA = \Spaces$, for which we now set up some terminology. An object $\Seg^n(\Spaces)$ is called an \emph{$n$-uple Segal space}. By unravelling the definition we see that an $n$-uple Segal space $X$ is an $n$-simplicial space $X : \bD^{n, \op} \to \Spaces$ satisfying the Segal condition in each variable independently. We call $X_{\undl{t}} := X([\undl{t}])$ the \emph{space of $\undl{t}$-morphisms} of $X$. It will be helpful to call \emph{elementary} those tuples $\undl{t}$ consisting only of $0$s and $1$s, and for elementary tuples their \emph{dimension} is the sum $\sum_{i} t_i$ (equivalently, the number of entries of $\undl{t}$ equal to $1$).
	
The Segal condition allows us to define composition operations among $\undl{t}$-morphisms in $X$. For $n = 1$, for example, we have a span
\begin{equation*}
	X_1 \times_{X_0} X_1 \xlongleftarrow{\simeq} X_2 \to X_1
\end{equation*} 
where the second map is induced by the active inclusion $[1] \cong \{0,2\} \hookrightarrow [2]$. We denote the resulting map $X_1 \times_{X_0} X_1 \to X_1$ by $- \circ_1 -$. Similarly, for $n \geq 2$ and any $i$ we have a span
\begin{equation*}
	X_{\bullet, \dotsc, 1, \dotsc \bullet} \times_{X_{\bullet, \dotsc, 0, \dotsc \bullet}} X_{\bullet, \dotsc, 1, \dotsc \bullet} \xlongleftarrow{\simeq} X_{\bullet, \dotsc, 2, \dotsc \bullet} \to X_{\bullet, \dotsc, 1, \dotsc \bullet}
\end{equation*}
induced by the same maps as before but in the $i$th factor of $\bD$, whose composite we denote by $- \circ_i -$. All these ``composition'' operations are appropriately unital and associative, and the associativity includes relations between compositions in different directions. The units are induced by the maps $\pi : [1] \to [0]$ in $\bD$; the operation of precomposing by $\pi$ in the $i$th factor is denoted by $\id_i(-)$ and increases the dimension of a morphism by $1$. We also have boundary operations, induced by the inclusions $\theta^0, \theta^1 : [0] \to [1]$; precomposing by $\theta^l$ in the $i$th factor is denoted by $\partial_i^l$, and this operation lowers the dimension by $1$. We call $\partial_i^0(f)$ and $\partial_i^1(f)$ the \emph{source and target of $f$ in the $i$th direction}, and we call $\id_i(f)$ the \emph{identity of $f$ in the $i$th direction}.

\subsection{Globularity}

We are going to introduce some extra conditions that one can impose on multisimplicial objects. Once again, we fix an $\infty$-category $\sA$ with finite limits.

\begin{lmm} \label{lmm:adj-lemma}
	Let $F : \sD \to \sE$ be a functor of $\infty$-categories with a right adjoint $G$, and let $\sD_0 \sub \sD$ and $\sE_0 \sub \sE$ be full sub-$\infty$-categories such that $F_0 := F \vert_{\sD_0}$ factors through $\sE_0$. If $G_0 := G\vert_{\sE_0}$ factors through $\sD_0$ then $G_0$ is a right adjoint of $F_0$.
\end{lmm}

\begin{proof}
	Let $d \in \sD_0$ and $e \in \sE_0$. Then
	\begin{equation*}
		\sE_0(F_0(d), e) \simeq \sE(F_0(d), e) \simeq \sD(d, G(e)) \simeq \sD(d, G_0(e)) \simeq \sD_0(d, G_0(e))
	\end{equation*}
	where the last equivalence holds because, by assumption, $G_0(e) \in \sD_0$.
\end{proof}

\begin{prp} \label{prp:ev-adjoints}
	Let $i : \bD^{n-1,\op} \simeq \bD^{n-1,\op} \times \{0\} \hookrightarrow \bD^{n,\op}$ denote the functor $[\undl{t}] \mapsto [\undl{t},0]$. The evaluation functor $\ev_0 := i^\ast : \Fun(\bD^{n,\op}, \sA) \to \Fun(\bD^{n-1,\op},\sA)$ given by $X \mapsto X_{\bullet, \dotsc, \bullet, 0}$ admits a left and a right adjoint: 
	\begin{itemize}
		\item[(L)] its left adjoint $i_! : \Fun(\bD^{n-1,\op}, \sA) \to \Fun(\bD^{n,\op}, \sA)$ sends $X$ to the functor $[\undl{t},u] \mapsto X_{\undl{t}}$ which is constant in the last variable;
		\item[(R)] its right adjoint $i_\ast : \Fun(\bD^{n-1,\op}, \sA) \to \Fun(\bD^{n,\op},\sA)$ sends $X$ to the functor $[\undl{t},u] \mapsto (X_{\undl{t}})^{\times (u+1)}$, with projections as face maps and diagonals as degeneracy maps.
	\end{itemize}
	The left adjoint is fully faithful. Moreover, the evaluation functor restricts to $\ev_0 : \Seg^n(\sA) \to \Seg^{n-1}(\sA)$ and the adjunctions $i_! \dashv \ev_0 \dashv i_\ast$ restrict to adjunctions at the level of higher Segal objects.
\end{prp} 

\begin{proof}
    The object $[0] \in \bD^\op$ is initial, so $\ev_0$ is simply the functor which takes the limit in the last factor of $\bD$; therefore (L) follows. From the formula we see that the unit of $i_! \dashv \ev_0$ is an equivalence, implying that $i_!$ is fully faithful. For (R) we just have to show that we can right Kan extend along $i$: indeed, the pointwise formula for right Kan extensions gives us
    \begin{equation*}
        (\mathrm{RKan}_i(X))_{\undl{t},u} \simeq \lim_{\substack{[\undl{t},u] \to [\undl{s},0] \\ \text{in } \bD^\op}} X_{\undl{s}} \simeq \lim_{\substack{[0] \to [u] \\ \text{in } \bD}} X_{\undl{t}} \simeq X_{\undl{t}}^{u + 1},
    \end{equation*}
    where the limit exists since we assumed that $\sA$ had all finite limits.

    The formulas show that $\ev_0 X$, $i_! X$, and $i_\ast X$ satisfy the Segal conditions in all coordinates whenever $X$ does, so all these functors restrict appropriately. The last claim follows using \Cref{lmm:adj-lemma} since $\Seg^n(\sA)$ is a full sub-$\infty$-category of $\Fun(\bD^\op, \sA)$.
\end{proof}

Since the left adjoint $i_!$ is fully faithful we have a diagram of inclusions of full sub-$\infty$-categories:
\begin{equation*}
	\begin{tikzcd}[column sep = small]
		\sA \ar[r,hook] \ar[d,equal] & \Seg^1(\sA) \ar[r,hook] \ar[d,hook] & \Seg^2(\sA) \ar[r,hook] \ar[d,hook] & \dotsb \ar[r,hook] & \Seg^{n-1}(\sA) \ar[r,hook] \ar[d,hook] & \Seg^{n}(\sA) \ar[d,hook] \\
		\sA \ar[r, hook] & \Fun(\bD^{\op}, \sA) \ar[r,hook] & \Fun(\bD^{2,\op}, \sA) \ar[r,hook] & \dotsb \ar[r,hook] & \Fun(\bD^{n-1,\op}, \sA) \ar[r,hook] & \Fun(\bD^{n,\op}, \sA)
	\end{tikzcd}
\end{equation*}

\begin{dfn}
	We say that a functor $X : \bD^{k,\op} \to \sA$ is \emph{constant} if it belongs to the sub-$\infty$-category $\sA$, or equivalently if the canonical map $X_{0, \dotsc, 0} \to X_{\undl{t}}$ is an equivalence for every $[\undl{t}] \in \bD^k$.
\end{dfn}

\begin{dfn}
	Let $1 \leq k \leq n$. We say that $X \in \Fun(\bD^{n,\op}, \sA)$ is \emph{globular at height $k$} if, for every $[\undl{t}] \in \bD^{k-1}$, the $(n-k-1)$-simplicial object $X_{\undl{t},0,\undl{\bullet}}$ is constant. We say that $X$ is \emph{globular up to height $k$} if it globular at heights $1, 2, \dotsc, k$. Denote by $\Glb^{n,k}(\sA) \sub \Fun(\bD^{n,\op}, \sA)$ the full sub-$\infty$-category spanned by the $n$-simplicial objects which are globular up to height $k$ and by $\GlbSeg^{n,k}(\sA)$ the further sub-$\infty$-category spanned by the $n$-uple Segal objects which are globular up to height $k$. If $X \in \Glb^{n,n}(\sA)$ we just say that it is \emph{globular}.
\end{dfn}

\begin{rmk}
	If $X \in \Seg^n(\sA)$, to check if it is globular at height $k$ it is enough to verify the constancy condition for elementary tuples $\undl{t}$, i.e. those containing only $0$s and $1$s.
\end{rmk}

If $X$ is globular up to dimension $k$ then it is globular up to dimension $k-1$, so we have a diagram of inclusions
\begin{equation*}
	\begin{tikzcd}[column sep = small]
		\GlbSeg^{n,n}(\sA) \ar[r,hook, "="'] \ar[d, hook] & \GlbSeg^{n,n-1}(\sA) \ar[r,hook] \ar[d, hook] & \GlbSeg^{n,n-2}(\sA) \ar[r,hook] \ar[d, hook] & \dotsb \ar[r,hook] & \GlbSeg^{n,1}(\sA) \ar[r,hook] \ar[d, hook] & \Seg^n(\sA) \ar[d, hook] \\
		\Glb^{n,n}(\sA) \ar[r, hook, "="'] & \Glb^{n,n-1}(\sA) \ar[r, hook] & \Glb^{n,n-2}(\sA) \ar[r, hook] & \dotsb \ar[r, hook] & \Glb^{n,1}(\sA) \ar[r, hook] & \Fun(\bD^{n,\op}, \sA)
	\end{tikzcd}
\end{equation*}
Note also that every $n$-simplicial object is globular at height $n$, since the $0$-uple Segal object $X_{\undl{t},0}$, being just objects of $\sA$, are always constant. Thus the top left and bottom left horizontal maps in the diagram are equalities.

After unpacking the definition, we see that an $n$-uple Segal space which is globular up to height $k$ is determined by the data of
\begin{itemize}
	\item spaces $X_{\undl{0^n}}$, $X_{1,\undl{0^{n-1}}}$, $X_{1,1,\undl{0^{n-2}}}$, $\dotsc$, $X_{\undl{1^k},\undl{0^{n-k}}}$ of \emph{$k$-morphisms} for $0 \leq k \leq i$,
	\item an $(n-k)$-uple Segal space $X_{\undl{1^{k}}, \bullet, \dotsc, \bullet}$ whose $(0,\dotsc,0)$-morphisms are the $k$-morphisms of $X$,
	\item appropriately unital and associative composition operations.
\end{itemize}

It turns out that globular $n$-uple Segal objects are also known by a different name. Consider the $\infty$-category $S^n(\sA)$ of \emph{$n$-fold Segal objects in $\sA$} (see \cite[Definition 4.4]{Haugseng2018}) defined inductively as follows:
\begin{itemize}
	\item $S^0(\sA) := \sA$,
	\item $S^1(\sA) := \Seg(\sA)$,
	\item $S^n(\sA) \sub \Seg(S^{n-1}(\sA))$ is the full subcategory spanned by those functors $X : \bD^\op \to S^{n-1}(\sA)$ such that $X_0$ is constant, i.e. the map $X_{0, \undl{0^{n-1}}} \to X_{0, \undl{t}}$ is an equivalence for any $[\undl{t}] \in \bD^{n-1}$.
\end{itemize}

\begin{prp} \label{prp:glb-is-cat}
	For any $n \geq 0$ there is an equivalence between $\GlbSeg^{n,n}(\sA)$ and the $\infty$-category of $n$-fold Segal objects in $\sA$, and this equivalence commutes with their inclusion into $\Seg^n(\sA)$.
\end{prp}

\begin{proof}
	Since both $\GlbSeg^{n,n}(\sA)$ and $S^n(\sA)$ are full sub-$\infty$-categories of $\Seg^n(\sA)$, it will be enough to show that $\GlbSeg^{n,n}(\sA)$ admits the same inductive definition as $S^n(\sA)$. The cases $n = 0$ and $n = 1$ are clear. For $n \geq 2$, an $X \in \Seg(\GlbSeg^{n-1,n-1}(\sA)) \sub \Fun(\bD^{n,\op}, \sA)$ has the property that the $(n-1)$-uple Segal object $X_1$ is globular at height $k$ for $1 \leq k \leq n-1$, which holds if and only if $X$ is globular at height $k$ for $2 \leq k \leq n$. By definition, $X_0$ is constant if and only if $X$ is globular at height $1$. Hence $X \in \GlbSeg^{n,n}(\sA)$ if and only if, as a Segal functor $X : \bD^{\op} \to \GlbSeg^{n-1,n-1}(\sA)$, the value $X_0$ is constant.
\end{proof}

We choose to avoid using the term ``$n$-fold Segal space'' in the rest of the paper, despite it being established in the literature. This is mostly because it can be easily confused with ``$n$-uple Segal space'' and because we feel that the adjective ``globular'' is a better descriptor for the condition satisfied by $n$-fold Segal spaces.

\subsection{Underlying globular objects} \label{sec:underlying}

In \cite[Proposition 4.12]{Haugseng2018} it is shown that the inclusion of $n$-fold Segal objects into $n$-uple Segal objects\footnote{A word of warning when checking the reference: Haugseng uses the notation $\Cat^n(\sA)$ for $n$-uple Segal objects and $\Seg^n(\sA)$ for $n$-fold Segal objects, while we use $\Seg^n(\sA)$ for the former and $\GlbSeg^{n,n}(\sA)$ for the latter.} admits a right adjoint $U^n$. Therefore, given \Cref{prp:glb-is-cat}, the inclusion $\GlbSeg^{n,n}(\sA) \sub \Seg^n(\sA)$ also admits a right adjoint. In this subsection we will produce a chain
\begin{equation*}
	\Fun(\bD^{n,\op},\sA) \to \Glb^{n,1}(\sA) \to \Glb^{n,2}(\sA) \to \dotsb \to \Glb^{n,n-1}(\sA) \to \Glb^{n,n}(\sA)
\end{equation*}
of right adjoints to the inclusions $\Glb^{n,k}(\sA) \sub \Glb^{n,k-1}$ whose composite is the right adjoint of the inclusion $\Glb^{n,n} \sub \Fun(\bD^{n,\op}, \sA)$; by \Cref{prp:glb-is-cat}, the latter is equivalent to the right adjoint $U^n$ of \cite{Haugseng2018}.

\begin{prp} \label{prp:rnk-adjoints}
	For every $n \geq 1$ and $1 \leq k \leq n$, the inclusion $\Glb^{n,k}(\sA) \sub \Glb^{n,k-1}(\sA)$ has a right adjoint $R_k^n : \Glb^{n,k-1}(\sA) \to \Glb^{n,k}(\sA)$.
\end{prp}

\begin{proof}
	We concentrate first on the case $k = 1$. The inclusion $\Glb^{n,1}(\sA) \sub \Seg^{n}(\sA)$ corresponds to that of those functors $\bD^{\op} \to \Seg^{n-1}(\sA)$ whose value at $0$ is constant, meaning we have a pullback square
	\begin{equation*}
		\begin{tikzcd}
			\Glb^{n,1}(\sA) \ar[r, hook] \ar[d, "{X \mapsto X_0}"'] & \Fun(\bD^{n,\op}, \sA) \ar[d, "{X \mapsto X_{0, \bullet, \dotsc, \bullet}}"] \\
			\sA \ar[r, hook] & \Fun(\bD^{n-1,\op}, \sA)
		\end{tikzcd}
	\end{equation*}
	The top map has a right adjoint $R_1^n$ by \cite[Lemma 4.14]{Haugseng2018}: indeed, the bottom map admits a right adjoint (\Cref{lmm:adj-lemma}), the right vertical map is a cartesian fibration by a slight modification of \cite[Lemma 4.15]{Haugseng2018}, and the diagram is a pullback.
	
	By induction on $n$ and $k$ we can assume that the inclusion $\Glb^{n-1,k-1}(\sA) \sub \Glb^{n-1,k-2}(\sA)$ admits a right adjoint $R_{k-1}^{n-1}$. The induced map $\Fun(\bD^\op, \Glb^{n-1,k-1}(\sA)) \to \Fun(\bD^\op, \Glb^{n-1,k-2}(\sA))$ admits a right adjoint $P$ sending $X : \bD^\op \to \Glb^{n-1,k-2}(\sA)$ to $R_{k-1}^{n-1}(X_\bullet) : \bD^\op \to \Glb^{n-1,k-1}(\sA)$ (\cite[Lemma 4.18]{Haugseng2018}). Thanks to \Cref{lmm:adj-lemma}, to show that $P$ restricts to a right adjoint
	\begin{equation*}
		R_{k}^n : \Glb^{n,k-1}(\sA) \to \Glb^{n,k}(\sA)
	\end{equation*}
	it will be enough to show that, for all $X : \bD^\op \to \Glb^{n-1,k-2}(\sA)$, if $X_0$ is constant then $P(X)_0$ is constant. But $P(X)_0 \simeq R^{n-1}_{k-1}(X_0) \simeq X_0$ because $X_0$ (being constant) is already globular up to height $k-1$, so $R^{n-1}_{k-1}$ fixes it; therefore $P(X)_0$ is constant and we're done.
\end{proof}

\begin{prp}
	The right adjoints $R_k^n : \Glb^{n,k-1}(\sA) \to \Glb^{n,k}(\sA)$ restrict to right adjoints $R_k^n : \GlbSeg^{n,k-1}(\sA) \to \GlbSeg^{n,k}(\sA)$ to the inclusion $\GlbSeg^{n,k}(\sA) \sub \GlbSeg^{n,k-1}(\sA)$.
\end{prp}

\begin{proof}
	The case $k = 1$ holds for the same reasons as in \Cref{prp:rnk-adjoints} since the constant diagram map $\sA \hookrightarrow \Fun(\bD^{n-1},\sA)$ factors through $\Seg^{n-1}(\sA)$ and since the map $\Seg^n(\sA) \to \Seg^{n-1}(\sA)$ evaluating at $[0]$ on the \emph{first} coordinate is a cartesian fibration (again by \cite[Lemma 4.15]{Haugseng2018}). Assume by induction that $R_{k-1}^{n-1}(X)$ satisfies the Segal condition whenever $X$ does. Then, if $X \in \GlbSeg^{n,k-1}(\sA) \sub \Fun(\bD^\op, \GlbSeg^{n-1,k-2}(\sA))$ is written as a Segal presheaf $Y : \bD^\op \to \GlbSeg^{n-1,k-2}(\sA)$, we have that $R_k^n(X)_{s,\undl{t}} \simeq R_{k-1}^{n-1}(Y_s)_{\undl{t}}$; thus $R_k^n(X)$ satisfies the Segal condition in the first coordinate since $R_{k-1}^{n-1}$ preserves all limits and in the remaining coordinates by the induction hypothesis. Hence $R_k^n$ restricts to a functor $\GlbSeg^{n,k-1}(\sA) \to \GlbSeg^{n,k}(\sA)$ and the claim follows by \Cref{lmm:adj-lemma}.
\end{proof}

\begin{crl} \label{crl:ra-composition}
	The composite
	\begin{equation*}
		R^n := R_n^n R_{n-1}^n \dotsm R_{2}^n R_1^n : \Seg^n(\sA) \to \GlbSeg^{n,1}(\sA) \to \dotsb \to \GlbSeg^{n,n}(\sA)
	\end{equation*}
	is a right adjoint to the inclusion $\Glb^{n,n}(\sA) \sub \Seg^n(\sA)$, and it is equivalent to the right adjoint $U^n$ under the identification of \Cref{prp:glb-is-cat}.
\end{crl}

We also have a relatively explicit formula for $R_k^n(X)$ in terms of $X$ generalizing \cite[Remark 4.13]{Haugseng2018}.

\begin{dfn}
	Let $\tau_k : \Fun(\bD^{n,\op}, \sA) \to \Fun(\bD^{k,\op}, \sA)$ denote the functor which evaluates at $[0]$ in the last $n-k$ coordinates: $(\tau_k X)_{\undl{t}} \simeq X_{\undl{t},\undl{0}}$ for any $[\undl{t}] \in \bD^k$. We call $\tau_k X$ the \emph{truncation of $X$} to height $k$.
\end{dfn}

\begin{rmk}
	Notice that $\tau_k$ restricts to truncations functors
	\begin{equation*}
		\Seg^n(\sA) \to \Seg^k(\sA), \quad \Glb^{n,l}(\sA) \to \Glb^{k,l'}(\sA), \quad \GlbSeg^{n,l}(\sA) \to \GlbSeg^{k,l'}(\sA)
	\end{equation*}
	where $l' = l$ if $l \leq n-k$ and $l' = k$ if $l \geq n-k$.
\end{rmk}

\begin{prp} \label{prp:ra-as-pullback}
	Fix $2 \leq k \leq n+1$ and let $X \in \Glb^{n,k-1}(\sA)$. Denote by $e : \Fun(\bD^{n,\op}, \sA) \to \Fun(\bD^{n-1,\op}, \sA)$ the functor evaluating the $k$th coordinate at $[0]$, and let $r$ be its right adjoint. Then the commutative square
	\begin{equation*}
		\begin{tikzcd}
			R_k^n(X) \ar[r] \ar[d] & X \ar[d] \\
			r (\tau_{k-1} X) \ar[r] & r (e X)
		\end{tikzcd}
	\end{equation*}
	in $\Fun(\bD^{n,\op},\sA)$ is a pullback.
\end{prp}

\begin{proof}
	First we will show that $P := X \times_{r(e X)} r (\tau_{k-1} X) \in \Glb^{n,k}(\sA)$ by showing that $P_{\undl{u},0,\undl{v}} \simeq P_{\undl{u},\undl{0^{n-k}}}$ for any $\undl{u} \in \bD^{k-1}$ and $\undl{v} \in \bD^{n-k-1}$. A quick computation shows that
	\begin{equation*}
		(r(\tau_{k-1} X))_{\undl{u},0,\undl{v}} \simeq X_{\undl{u},\undl{0^{n-k}}}, \qquad (r(eX))_{\undl{u},0,\undl{v}} \simeq X_{\undl{u},0,\undl{v}}
	\end{equation*}
	and so
	\begin{equation*}
		P_{\undl{u},0,\undl{v}} \simeq X_{\undl{u},0,\undl{v}} \times_{X_{\undl{u},0,\undl{v}}} X_{\undl{u},\undl{0^{n-k}}} \simeq X_{\undl{u},\undl{0^{n-k}}} \simeq P_{\undl{u},\undl{0^{n-k}}}.
	\end{equation*}
	
	Now note that if $Y \in \Glb^{n,k}(\sA)$ then
	\begin{equation*}
		\Map(Y, r(\tau_{k-1} X)) \simeq \Map(e Y, \tau_{k-1} X) \simeq \Map(\tau_{k-1} Y, \tau_{k-1} X)
	\end{equation*}
	and
	\begin{equation*}
		\Map(Y, r(e X)) \simeq \Map(e Y, e X) \simeq \Map(\tau_{k-1} Y, e X) \simeq \Map(\tau_{k-1} Y, \tau_{k-1} X).
	\end{equation*}
	since $eY \simeq \tau_{k-1} Y$ by the globularity of $Y$ at height $k$. The equivalences are seen to commute with the map $r(\tau_{k-1} X) \to r(eX)$, so
	\begin{equation*}
		\Map(Y,P) \simeq \Map(Y,X) \times_{\Map(Y, r(eX))} \Map(Y, r(\tau_{k-1} X)) \simeq \Map(Y,X)
	\end{equation*}
	naturally in $Y \in \Glb^{n,k}(\sA)$, which implies that $P \simeq R_k^n(X)$.
\end{proof}

\subsection{Segalification} \label{sec:segalification}

From now on we will focus on the case $\sA = \Spaces$ which is central to the aim of this paper; nevertheless, most of what will appear can be generalized to the case of an arbitrary $\infty$-topos.

Recall that since $\Seg(\Spaces)$ is obtained from the presheaf $\infty$-category $\Fun(\bD^{\op}, \Spaces)$ by localizing at the Segal maps, $\Seg(\Spaces)$ is a \emph{reflective sub-$\infty$-category} of $\Fun(\bD^\op, \Spaces)$, meaning that there is a left adjoint $\LL : \Fun(\bD^\op, \Spaces) \to \Seg(\Spaces)$ to the inclusion. We call $\LL$ the \emph{Segalification functor}. The inclusion $\Seg^n(\Spaces) \sub \Fun(\bD^{n,\op}, \Spaces)$ is also reflective and so comes with its own Segalification functor $\LL : \Fun(\bD^{n,\op}, \Spaces) \to \Seg^n(\Spaces)$.

\begin{dfn}
	The \emph{Segalification functor in the $k$th coordinate} is the functor $\LL_k : \Fun(\bD^{n,\op}, \Spaces) \to \Fun(\bD^{n-1,\op}, \Seg(\Spaces))$ defined by
	\begin{equation*}
		\Fun(\bD^{n,\op}, \Spaces) \simeq \Fun(\bD^{n-1,\op}, \Fun(\bD^\op, \Spaces)) \xlongrightarrow{\LL \circ -} \Fun(\bD^{n-1,\op}, \Spaces)
	\end{equation*}
	where the first equivalence moves the $k$th copy of $\bD^\op$ to the right-hand side of $\Fun(-,-)$.
\end{dfn}

Note that we can apply $\LL_k$ iteratively, one $k$ at a time, to obtain a functor $\Fun(\bD^{n,\op}, \Spaces) \to \Fun(\bD^{n,\op}, \Spaces)$. This functor is \emph{not equivalent} to $\LL : \Fun(\bD^{n,\op}, \Spaces) \to \Fun(\bD^{n,\op}, \Spaces)$. Indeed, for any $k \neq l$, $\LL_l \LL_k X$ is Segal in the $l$th coordinate but is not guaranteed to be Segal in the $k$th coordinate -- this is essentially because $\LL_l$ might not preserve finite limits. There is a partial fix for this:
\begin{fct}[{\cite[Lemma 2.15]{BS2024}}]
	Assume that $X \in \Glb^{n,n}(\Spaces)$ satisfies the Segal condition in the first $k-1$ coordinates. Then $\LL_k X$ remains globular (i.e. $\LL_k X \in \Glb^{n,n}(\Spaces)$) and satisfies the Segal condition in the first $k$ coordinates.
\end{fct}

Thus if we Segalify a globular presheaf one coordinate at a time we don't have to worry about the Segal condition breaking at any point. The fix can be strengthened to presheaves that are globular only up to height $k$; in fact, the following result was the main drive to develop the definition of $\Glb^{n,k}(-)$ in the first place.

\begin{prp} \label{prp:fix-segalification}
	Assume $X \in \Glb^{n,k-1}(\Spaces)$ satisfies the Segal condition in the first $k-1$ coordinates. Then $\LL_k X$ remains in $\Glb^{n,k-1}(\Spaces)$ and satisfies the Segal condition in the first $k$ coordinates.
\end{prp}

An immediate consequence is:

\begin{crl} \label{crl:alternate-segalification}
	If $X \in \Fun(\bD^{n,\op}, \Spaces)$ then
	\begin{equation*}
		(R_n^n \circ \LL_n \circ R_{n-1}^n \circ \LL_{n-1} \circ \dotsb \circ R_1^n \circ \LL_1) (X)
	\end{equation*}
	is a globular $n$-uple Segal space.
\end{crl}

The proof of \Cref{prp:fix-segalification} is a straightforward modification of the proof of \cite[Lemma 2.15]{BS2024}; we will write it down anyway, both for ease of reference and because we will make use of the intermediary \Cref{fct:formula-segalification} in later parts of the paper.

\begin{dfn}
	The \emph{wedge} $A \vee B$ of two bipointed categories $(A, a_0, a_1)$ and $(B, b_0, b_1)$ is the bipointed category $(A \sqcup_{a_1 = b_0} B, a_0, b_1)$. A \emph{necklace} is an iterated wedge $N = [n_1] \vee [n_2] \vee \dotsb \vee [n_k]$ where each representable $[n]$ carries the bipointing $(0,n)$. We denote by $\bN$ the category whose objects are necklaces and whose morphisms are maps of bipointed categories.
\end{dfn}

\begin{fct}[{\cite[Observation 1.16, Proposition 1.17]{BS2024}}] \label{fct:formula-segalification}
	The Segalification functor $\LL : \Fun(\bD^\op, \Spaces) \to \Seg(\Spaces)$ can be computed pointwise as
	\begin{equation*}
		(\LL X)_t \simeq \colim_{(N_1, \dotsc, N_t) \in \bN^{n, \op}} \Map(N_1 \vee N_2 \vee \dotsb \vee N_t, X).
	\end{equation*}
	In particular, $\LL(X)_0 \simeq X_0$ and the space of $1$-morphisms of $\LL X$ is
	\begin{equation*}
		(\LL X)_1 \simeq \colim_{N \in \bN^\op} \Map(N, X) \simeq \colim_{[n_1] \vee \dotsb \vee [n_k]} X_{n_1} \times_{X_0} \dotsb \times_{X_0} X_{n_k} 
	\end{equation*}
	and so any $1$-morphism of $\LL X$ can be written (possibly non-uniquely) as a formal composition of points of $X_1$.
\end{fct}

\begin{fct}[{\cite[Lemmas 1.27, 2.14]{BS2024}}] \label{fct:sifted-colim}
	Suppose $F : \bN^\op \to \Seg(\Spaces)$ is a functor such that $F(-)_0 : \bN^\op \to \Spaces$ is constant. Then $\displaystyle \colim_{N \in \bN^\op} F(N)$ is a Segal space. 
\end{fct}

\begin{proof}[{Proof of \Cref{prp:fix-segalification}}]
	If $i \leq k-2$ and $[\undl{t}] \in \bD^i$ then $(\LL_k X)_{\undl{t},0,\undl{\bullet}} \simeq \LL_{k-i-1} (X_{\undl{t},0,\undl{\bullet}})$; but $X_{\undl{t},0,\undl{\bullet}}$ is a constant $(n-i-1)$-simplicial space (since $X$ is globular up to height $k-1$) and hence an $(n-i-1)$-uple Segal space, so it is fixed under Segalification. Therefore $(\LL_k X)_{\undl{t},0,\undl{\bullet}} \simeq X_{\undl{t},0,\undl{\bullet}} \simeq X_{\undl{t}, \undl{0^{n-i}}} \simeq (\LL_k X)_{\undl{t},\undl{0^{n-i}}}$ and so $\LL_k X$ is also globular up to height $k-1$.
	
	By construction $\LL_k X$ is Segal in the $k$th coordinate. To prove it remains Segal in the $i$th coordinate, for $i < k$, we just have to show that $(\LL_k X)_{\undl{t},\bullet,\undl{u},l,\undl{v}}$ is a Segal space when $l = 0, 1$, where we have fixed $[\undl{t}] \in \bD^{i-1}$, $[\undl{u}] \in \bD^{k-i-1}$, and $[\undl{v}] \in \bD^{n-k}$. For $l = 0$ there is nothing to do: $\LL_k$ fixes the space of objects (\Cref{fct:formula-segalification}) and so $(\LL_k X)_{\undl{t},\bullet,\undl{u},0,\undl{v}} \simeq X_{\undl{t},\bullet,\undl{u},0,\undl{v}}$ is clearly Segal. For $l = 1$, \Cref{fct:formula-segalification} gives us
	\begin{equation*}
		(\LL_k X)_{\undl{t},\bullet,\undl{u},1,\undl{v}} \simeq \colim_{[n_1] \vee \dotsb \vee [n_j] \in \bN^\op} X_{\undl{t},\bullet,\undl{u},n_1,\undl{v}} \times_{X_{\undl{t},\bullet,\undl{u},0,\undl{v}}} \dotsb \times_{X_{\undl{t},\bullet,\undl{u},0,\undl{v}}} X_{\undl{t},\bullet,\undl{u},n_j,\undl{v}}
	\end{equation*}
	as presheaves on $\bD$. We can now apply \Cref{fct:sifted-colim}: each $X_{\undl{t},\bullet,\undl{u},l,\undl{v}}$ is Segal by assumption (as $i < k$) and the Segal condition is preserved under pullbacks, so each $X_{\undl{t},\bullet,\undl{u},n_1,\undl{v}} \times_{X_{\undl{t},\bullet,\undl{u},0,\undl{v}}} \dotsb \times_{X_{\undl{t},\bullet,\undl{u},0,\undl{v}}} X_{\undl{t},\bullet,\undl{u},n_j,\undl{v}}$ is a Segal space; moreover, plugging $\bullet = 0$ yields
	\begin{equation*}
		X_{\undl{t},0,\undl{u},n_1,\undl{v}} \times_{X_{\undl{t},0,\undl{u},0,\undl{v}}} \dotsb \times_{X_{\undl{t},0,\undl{u},0,\undl{v}}} X_{\undl{t},0,\undl{u},n_j,\undl{v}} \simeq X_{\undl{t}, \undl{0^{n-i}}} \times_{X_{\undl{t}, \undl{0^{n-i}}}} \dotsb \times_{X_{\undl{t}, \undl{0^{n-i}}}} X_{\undl{t}, \undl{0^{n-i}}} \simeq X_{\undl{t}, \undl{0^{n-i}}}
	\end{equation*}
	since $X$ is globular at height $i$ for $i \leq k-1$. The fact implies that the colimit is again a Segal space, so we're done.
\end{proof}

\subsection{A brief comment on completeness} \label{sec:completeness}

There is an additional condition that one can impose on globular $n$-uple Segal spaces, called \emph{completeness}, which we won't dwell much on. They key property we need is that the full sub-$\infty$-category of $\GlbSeg^{n,n}(\Spaces)$ spanned by the complete objects is a model of the $\infty$-category of $(\infty,n)$-categories; this was proven for $n = 1$ in \cite{JT2007} and for arbitrary $n$ in \cite{Barwick2005}. In other words, we have a fully faithful functor $\fCat_{(\infty,n)} \hookrightarrow \GlbSeg^{n,n}(\Spaces)$ with a left adjoint $C : \GlbSeg^{n,n}(\Spaces) \to \fCat_{(\infty,n)}$ called the \emph{completion functor}. This is good news for us since the $(\infty,n)$-category we are interested in should have a mapping-out universal property, so we can reasonably construct it as $C(X)$ for some $X \in \GlbSeg^{n,n}(\Spaces)$. Then, instead of studying $C(X)$ closely, we can study the space $\Map(C(X), \sD)$ of functors into an $(\infty,n)$-category $\sD$ by studying the (possibly simpler) equivalent space $\Map(X, \sD)$.

In light of this equivalence, from now on we will use the term ``$(\infty,n)$-category'' to mean ``complete globular $n$-uple Segal space''. We will also treat every strict $n$-category as an $(\infty,n)$-category and thus as a globular $n$-uple Segal space.

\begin{rmk} \label{rmk:walking-morphisms}
	There is a functor $w : \bD^n \to \fCat_{n}$ which sends $[\undl{t}]$ to the \emph{walking $\undl{t}$-cell}. It is uniquely determined by the following three properties:
	\begin{enumerate}[label=(\alph*)]
		\item $w(k) := w([1, \dotsc, 1, 0, \dotsc, 0])$ is the walking $k$-morphism;
		\item if $\undl{t} = (\undl{r},0,\undl{s})$ then $w(\undl{t}) = w(\undl{r}, \undl{0})$;
		\item $w$ is co-Segal: it sends the Segal maps to pushouts. 
	\end{enumerate}
	It can be verified that the composite $w : \bD^n \to \fCat_{n} \sub \fCat_{(\infty,n)} \to \GlbSeg^{n,n}(\Spaces)$ satisfies
	\begin{equation*}
		\Map(w(\undl{t}), X) \simeq \Map([\undl{t}], X) \simeq X_{\undl{t}}
	\end{equation*}
    for any globular $n$-uple Segal space $X$.
\end{rmk} 

\subsection{Adjoint morphisms}

We conclude this section with a few definitions.

\begin{dfn} \label{dfn:adj}
	Denote by $\Adj$ the universal $2$-category containing an adjunction $f \dashv g$ (see, for example, \cite{RV2016}). Explicitly, $\Adj$ is generated by two objects $x, y$, two $1$-morphisms $f : x \to y$, $g : y \to x$, and two $2$-morphism $\eta : \id(x) \Rightarrow gf$, $\varepsilon : fg \Rightarrow \id(y)$ satisfying the snake equations
	\begin{equation*}
		(\varepsilon \circ_1 \id(f)) \circ_2 (\id(f) \circ_1 \eta) = \id(f), \qquad (\id(g) \circ_1 \varepsilon) \circ_2 (\eta \circ_1 \id(g)) = \id(g).
	\end{equation*}
	As usual, we call $f$ the \emph{left adjoint}, $g$ the \emph{right adjoint}, $\eta$ the \emph{unit}, and $\varepsilon$ the \emph{counit}. An \emph{adjunction of $1$-morphisms} in an $(\infty,n)$-category $\sC$ is a map $\Adj \to \sC$. 
\end{dfn}

\begin{dfn}
	Let $\sC$ be an $(\infty,n)$-category and let $1 \leq k \leq n-1$. An \emph{adjunction of $k$-morphisms} in $\sC$ is an adjunction of $1$-morphisms in $\mathrm{Mor}^{k-1}(\sC)$, the $(\infty,n-k+1)$-category of $(k-1)$-morphisms in $\sC$. If $X \in \GlbSeg^{nn,}(\Spaces)$ has completion $\mathscr{X}$, an \emph{adjunction of $k$-morphisms} in $X$ is an adjunction of $k$-morphisms in $\mathscr{X}$. 
\end{dfn}

\begin{rmk}
	If $X$ is not complete we can still compute the adjunctions in it without first completing it: they are given as in \Cref{dfn:adj} except that the snake equations only have to hold up to an invertible higher morphism. If $X$ is complete then higher invertible morphisms are equivalent to identities, so the problem of distinguishing the two definitions does not arise.
\end{rmk}

We borrow the following terminology from \cite{DPP2003}:

\begin{dfn}
	Let $X \in \GlbSeg^{n,n}(\Spaces)$. We say that $X$ is \emph{k-sinister} if every $k$-morphism of $X$ extends to an adjunction of $k$-morphisms in $X$ for which it is a left adjoint. We say that $X$ is \emph{sinister} if it is $k$-sinister for all $k = 1, \dotsc, n-1$.
\end{dfn}

\section{Squares and zigzags} \label{sec:squares}

In this section we will define our main object of interest, a functor $\zig_{+}^{n+1} : \GlbSeg^{n,n}(\Spaces) \to \GlbSeg^{n+1,n+1}(\Spaces)$ such that the morphisms of $\zig_{+}^{n+1}(X)$ are certain zigzag diagrams in $X$. First we will have to introduce an auxiliary functor $\Sq^n : \GlbSeg^{n,n}(\Spaces) \to \Seg^{n}(\Spaces)$, the \emph{higher square functor}, which we then modify a bit to obtain $\zig_{+}^{n+1}$.

\subsection{Lax cubes} \label{sec:lax-cubes}

We borrow the following result from \cite[Theorem A]{Campion2023}:
\begin{prp} \label{prp:lax-gray-cubes}
	For any $n \geq 1$ there is a functor $- \otimes - : \fCat_{(\infty,n)} \times \fCat_{(\infty,n)} \to \fCat_{(\infty,n)}$ which turns $\fCat_{(\infty,n)}$ into a closed monoidal $\infty$-category with unit $[0]$, the terminal $(\infty,n)$-category.
\end{prp}

\begin{dfn}
	Denote by $\square^n : \bD^n \to \fCat_{(\infty,n)}$ the composite
	\begin{equation*}
		y \otimes \dotsb \otimes y : \bD^n \hookrightarrow (\fCat_{(\infty,1)})^n \hookrightarrow (\fCat_{(\infty,n)})^n \to  \fCat_{(\infty,n)}
	\end{equation*}
	of the $n$-fold product of the Yoneda embedding $y$ with the tensor product $\otimes$ from Campion's theorem. We call $\square(n) := \square^n(1, \dotsc, 1) \in \fCat_{(\infty,n)}$ the \emph{$n$-dimensional lax cube}.
\end{dfn}

\begin{lmm} \label{lmm:square-segal}
	The functor $\square^n$ has the following properties:
	\begin{enumerate}
		\item it factors through gaunt strict $n$-categories;
		\item it is co-Segal: its values are given by iterated pushouts of the values of elementary $n$-tuples, and these pushouts can be computed strictly or weakly;
		\item for any $a, b \geq 0$ such that $a + b = n+1$ there is a commutative diagram
		\begin{equation*}
			\begin{tikzcd}
				\bD^{n-1} \ar[d, "\simeq"'] \ar[r, "\square^{n-1}"] & \fCat_{(\infty,n-1)} \ar[dd, hook] \\
				\bD^{a} \times \{[0]\} \times \bD^b \ar[d, hook] & \\
				\bD^n \ar[r, "\square^n"'] & \fCat_{(\infty,n)}
			\end{tikzcd}
		\end{equation*}
	\end{enumerate}
\end{lmm}

\begin{proof}
	The first property follows from the fact that $\otimes$ restricts to the usual Gray tensor product of strict $n$-categories. The second property holds since the Yoneda embedding $y : \bD \to \fCat_{(\infty,1)}$ is co-Segal and $\otimes$ preserves colimits in each variable; that the pushouts can be computed strictly is spelled out in detail in \cite{Campion2023}. The third property is a reformulation of the fact that $[0]$ is the unit of $\otimes$.
\end{proof}

Our next order of business is to understand $\square(n)$ for all $n$; thanks to the two properties above this is enough to understand $\square^n(\undl{t})$ for all $\undl{t}$. The description of $\square(n)$ that we will use is a small modification of the one presented in \cite{ABS2002} with the only difference occurring in the sign rule that determines the orientation of the boundary of a given $k$-morphism of $\square(n)$. This will be made clearer once we have defined the objects of interest and stated the characterization.

We fix a combinatorial $n$-cube $I^n$, the $n$-fold product of the combinatorial interval $I = \{0, 1\}$. Let $p^0 = \{0\}$ and $p^1 = \{1\}$.

\begin{dfn}
	A \emph{$k$-dimensional face $f$ of $I^n$} is a word $a_1 a_2 \dotsb a_n$ of length $n$ in the alphabet $\{p^0, p^1, I\}$ containing exactly $k$ many copies of $I$. For two faces $f = a_1 a_2 \dotsb a_n$ and $g = b_1 b_2 \dotsb b_n$, we say that \emph{$f$ is contained in $g$} and write $f \sub g$ if $a_i \neq b_i$ implies $b_i = I$.
\end{dfn}

\begin{rmk} \label{rmk:faces}
	The characteristic map of a face $f$ is the injective function $\chi_f : I^k \to I^n$ defined as the product $\vphi_1 \times \vphi_2 \times \dotsb \times \vphi_n$, where
	\begin{equation*}
		\vphi_i =
		\begin{cases}
			p^l \hookrightarrow I & \text{if }a_i = p^l, \\
			\id_I & \text{if } a_i = I
		\end{cases}
	\end{equation*}
	is either an inclusion or the identity map on $I$. Geometrically, a face $f = a_1 a_2 \dotsm a_n$ with $a_i \in \{p^0, p^1, I\}$ corresponds to the subset $\prod_i a_i \sub I^n$, which is also the image of $\chi_f$.
\end{rmk}

\begin{dfn} \label{dfn:parity}
	We define the \emph{parity $\sigma(f,g)$ of $f$ relative to $g$} for a $k$-dimensional face $f$ contained in a $(k+1)$-dimensional face $g$ of $I^n$.
	\begin{itemize}
		\item If $k = n-1$ then $g = I^n$ and $f$ contains exactly one copy of $p^l$, say in coordinate $r$, for some $l \in \{0,1\}$. The parity of $f$ is the number $\sigma(f, I^n) := r + l + n$ computed modulo $2$.
		\item If $k < n-1$ then the characteristic map of $f$ factors through that of $g$ via a function $I^k \to I^{k+1}$. In this case we set $n = k+1$ and treat $g$ as $I^{n}$, $f$ as a face of $g$ of dimension $n - 1 = k$, and apply the previous definition.
	\end{itemize}
\end{dfn}

\begin{rmk}
	An equivalent way to compute $\sigma(f,g)$ is as follows. Take the word $g = b_1 \dotsb b_n$ and erase all $n - k - 1$ copies of $p^l$, leaving only the $k+1$ copies of $I$. Say that those copies of $p^l$ were $b_{i_1}, \dotsc, b_{i_{n-k-1}}$. Now erase from $f = a_1 \dotsb a_n$ the entries $a_{i_1}, \dotsc, a_{i_{n-k-1}}$, which were all copies of $p^l$ since $f \sub g$. Now the resulting word has length $k+1$ and contains $k$ copies of $I$ and a single $p^l$, say in the new coordinate $r$. The number $\sigma(f,g)$ is precisely $r + l + (k + 1)$.
\end{rmk}

\begin{fct}[{Reformulation of \cite[Theorem 1.3]{ABS2002}}] \label{prp:cube-structure}
	For any $n$, $\square(n)$ is a strict gaunt finite $n$-category with the following properties:
	\begin{enumerate}[label=(\alph*)]
		\item if $G^n_k$ denotes the set of $k$-dimensional faces of $I^n$ then there is an injection from $G^n_k \hookrightarrow M^n_k$ to the set of non-degenerate $k$-morphisms of $\square(n)$;
		\item if $f \in G^n_k$, the source/target of $f$ when treated as a $k$-morphism of $\square(n)$ is the (unique) composition of the even/odd $(k-1)$-dimensional faces of $f$, where the parity is calculated relative to $f$;
		\item the subset $G^n_k$ and the identity $k$-morphisms of lower dimensional morphisms freely generate $M^n_k$ under composition, where freely means that the only relation is the cancellation of identities.
	\end{enumerate}
\end{fct}

\begin{rmk}
	As mentioned before, the only change we made from the construction in \cite{ABS2002} is in how the parity of a face is defined; translating the definition given right before their Theorem 1.3 we see that they set it to be $\sigma(f,I^n) = r + l + 1$, without the contribution of the dimension $n$. Thus our parity corresponds with theirs when $n$ is odd and is opposite to theirs when $n$ is even. The parity only changes the direction of the even-dimensional morphisms, so one can pass from one model to the other by applying $(-)^{\op_2, \op_4, \dotsc}$. We will explain later why our definition works better for our purposes.
\end{rmk}

\begin{exm}
	Here are the first three values of $\square(n)$. These can be computed from \Cref{prp:cube-structure} or looked up -- for example, they appear in \cite[Figure 1]{Campion2023}.
	\begin{enumerate}[start = 0]
		\item $\square(0)$ is the terminal category $[0]$, which we draw like this: $\bullet$. 
		\item $\square(1)$ is the walking arrow $[1]$, which we draw like this: $\begin{tikzcd}[sstyle] \bullet \ar[r] & \bullet \end{tikzcd}$.
		\item $\square(2)$ is the walking lax commutative square, which we draw like this: $\begin{tikzcd}[sstyle] \bullet \ar[r] \ar[d] & \bullet \ar[d] \ar[dl, phantom, "{\rotatebox{-45}{$\Downarrow$}}" description] \\ \bullet \ar[r] & \bullet \end{tikzcd}$.
		\item We draw $\square(3)$ as two $2$-categories connected by a $3$-morphism:
		
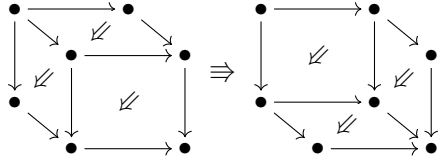
\begin{equation*}
			\begin{tikzcd}[sstyle]
				\bullet \ar[rr] \ar[dd] \ar[dr] & & \bullet \ar[dr] \ar[dl, phantom, "{\rotatebox{-45}{$\Downarrow$}}" description] & \\
				& \bullet \ar[rr] \ar[dd] \ar[dl, phantom, "{\rotatebox{-45}{$\Downarrow$}}" description] & & \bullet \ar[dd] \ar[ddll, phantom, "{\rotatebox{-45}{$\Downarrow$}}" description] \\
				\bullet \ar[dr] & & & \\
				& \bullet \ar[rr] & & \bullet
			\end{tikzcd}
			\, \Rrightarrow \,
			\begin{tikzcd}[sstyle]
				\bullet \ar[rr] \ar[dd] & & \bullet \ar[dr] \ar[dd] \ar[ddll, phantom, "{\rotatebox{-45}{$\Downarrow$}}" description] & \\
				& & & \bullet \ar[dd] \ar[dl, phantom, "{\rotatebox{-45}{$\Downarrow$}}" description] \\
				\bullet \ar[dr] \ar[rr] & & \bullet \ar[dr] \ar[dl, phantom, "{\rotatebox{-45}{$\Downarrow$}}" description] & \\
				& \bullet \ar[rr] & & \bullet
			\end{tikzcd}
		\end{equation*}
	\end{enumerate}
\end{exm}

\begin{crl}
	Let $F(n)$ denote the poset of $k$-dimensional faces of $I^n$ for $0 \leq k < n$, ordered by containment. Then there is a functor $F(n) \to \fCat_{(\infty,n)}$, $f \mapsto \square(\dim f)$ whose colimit $\partial \square(n) \in \fCat_{(\infty,n)}$ is the underlying sub-$(n-1)$-category of $\square(n)$. In particular we have a colimit decomposition
	\begin{equation*}
		\square(n) \simeq \partial \square(n) \sqcup_{w(n-1) \sqcup w(n-1)} w(n)
	\end{equation*}
	where $w(n)$ is the walking $n$-morphism.
\end{crl}

\begin{proof}
	This is immediate from the fact that the $k$-dimensional faces of $I^n$ generate all the $k$-morphisms of $\square(n)$, and that the compositions are all free (i.e. obtained by pasting walking $k$-morphisms together).
\end{proof}

\subsection{Some maps of lax cubes}

The advantage of having an explicit description of $\square(n)$ is that mapping out of it becomes easier. In this subsection we will study ``collapse'' maps $\kappa(n) : \square(n) \to w(n)$ onto the walking $n$-morphism.

First recall that $w(n)$ has exactly two non-degenerate parallel $k$-morphisms for $0 \leq k < n$ and a single $n$-morphism. Label the unique $n$-morphism by $J^n$ and label the two $k$-morphisms via the words $J^k q^l$; with respect to the (one or two) $(k+1)$-morphism(s) in $w(n)$, the one with $l = 0$ is the source and the one with $l = 1$ is the target. Note that the source and target of $J^k q^l$ are given by $J^{k-1} q^0$ and $J^{k-1} q^1$, respectively. Now, for a face $f$ of $I^n$ of dimension smaller than $n$ there exists a unique pair $(k_f,l_f)$ such that the word representing $f$ begins with $I^{k_f} p^{l_f}$. Note that $\dim f \geq k_f$ since there could be some more copies of $I$ in the word representing $f$ that appear after $p^{l_f}$. 

\begin{dfn}
	Define $\kappa(n)$ on a generator $f$ of $\square(n)$ by 
	\begin{equation*}
		\kappa(n)(f) := \id^{\dim f - k_f} (J^{k_f} q^{l_f})
	\end{equation*}
	if $\dim f < n$ and $\kappa(n)(I^n) := J^n$. Composition in $\square(n)$ is free so we are forced to define $\kappa(n)(f \circ_k g) := \kappa(n)(f) \circ_k \kappa(n)(g)$ for arbitrary morphisms $f, g$ of $\square(n)$.
\end{dfn}

\begin{exm} \label{exm:kappa}
	It might be instructive to draw what happens when $n = 2$ and $n = 3$ ($\kappa(n)$ is the identity when $n \leq 1$). This is $\kappa(2)$: 
	
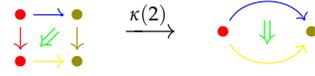
\begin{equation*}
		\begin{tikzcd}[sstyle]
			{\color{red} \bullet} \ar[r, color = blue] \ar[d, color = red] & {\color{olive} \bullet} \ar[d, color = olive] \ar[dl, phantom, "\rotatebox{-45}{$\Downarrow$}" description, color = green] \\
			{\color{red} \bullet} \ar[r, color = yellow] & {\color{olive} \bullet}
		\end{tikzcd}
		\quad \xlongrightarrow{\kappa(2)} \quad
		\begin{tikzcd}[sstyle]
			{\color{red} \bullet} \ar[rr, bend left = 60, color = blue] \ar[rr, bend right = 60, color = yellow] \ar[rr, phantom, "\Downarrow" description, color = green] & & {\color{olive} \bullet}
		\end{tikzcd}
	\end{equation*}
	We interpret the colors as telling us where each component is sent to. For example, the red colored dots and arrow on the left are all sent to the red dot on the right, and the top blue arrow on the left is sent to the top blue arrow on the right.
	
	This is $\kappa(3)$:
	
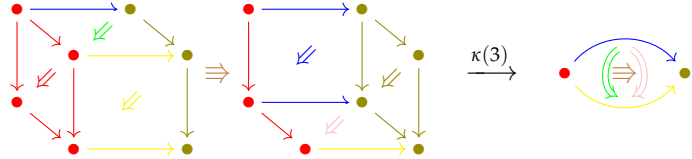
\begin{equation*}
		\begin{tikzcd}[sstyle]
			{\color{red} \bullet} \ar[rr, color = blue] \ar[dd, color = red] \ar[dr, color = red] & & {\color{olive} \bullet} \ar[dr, color = olive] \ar[dl, phantom, "{\rotatebox{-45}{$\Downarrow$}}" description, color = green] & \\
			& {\color{red} \bullet} \ar[rr, color = yellow] \ar[dd, color = red] \ar[dl, phantom, "{\rotatebox{-45}{$\Downarrow$}}" description, color = red] & & {\color{olive} \bullet} \ar[dd, color = olive] \ar[ddll, phantom, "{\rotatebox{-45}{$\Downarrow$}}" description, color = yellow] \\
			{\color{red} \bullet} \ar[dr, color = red] & & & \\
			& {\color{red} \bullet} \ar[rr, color = yellow] & & {\color{olive} \bullet}
		\end{tikzcd}
		\, {\color{brown} \Rrightarrow} \,
		\begin{tikzcd}[sstyle]
			{\color{red} \bullet} \ar[rr, color = blue] \ar[dd, color = red] & & {\color{olive} \bullet} \ar[dr, color = olive] \ar[dd, color = olive] \ar[ddll, phantom, "{\rotatebox{-45}{$\Downarrow$}}" description, color = blue] & \\
			& & & {\color{olive} \bullet} \ar[dd, color = olive] \ar[dl, phantom, "{\rotatebox{-45}{$\Downarrow$}}" description, color = olive] \\
			{\color{red} \bullet} \ar[dr, color = red] \ar[rr, color = blue] & & {\color{olive} \bullet} \ar[dr, color = olive] \ar[dl, phantom, "{\rotatebox{-45}{$\Downarrow$}}" description, color = pink] & \\
			& {\color{red} \bullet} \ar[rr, color = yellow] & & {\color{olive} \bullet}
		\end{tikzcd}
		\quad \xlongrightarrow{\kappa(3)} \quad
		\begin{tikzcd}[sstyle]
			{\color{red} \bullet} \ar[rrr, ""'{name=T}, bend left = 50, color = blue] \ar[rrr, ""{name=B}, bend right = 50, color = yellow] & & & {\color{olive} \bullet}
			\ar[from=T, to=B, Rightarrow, bend right = 60, ""{name=C}, color = green]
			\ar[from=T, to=B, Rightarrow, bend left = 60, ""'{name=A}, color = pink]
			\ar[from=C, to=A, phantom, "\Rrightarrow", color = brown]
		\end{tikzcd}
	\end{equation*}
	Obviously, in both examples, whenever a $k$-morphism $f$ on the left has the same color as a $j$-morphism $g$ on the right and $j < k$ then we are indicating that $f$ is sent to the corresponding identity $k$-morphism of $g$. 
\end{exm}

\begin{prp}
	The associations above make $\kappa(n)$ into a functor $\square(n) \to w(n)$ of $n$-categories.
\end{prp}

\begin{proof}
	Composition is taken care of, as are identities, so we just have to check that $\kappa(n)$ preserves the sources and targets of the generating morphisms. We will only treat the case where $\dim f = k_f$, as the general case can be deduced from it. If $\dim f = k_f$ then $f$ starts with $I^{k_f} p^{l_f}$ and contains no other copies of $I$. Then $\kappa(n)(f) = J^{k_f} q^{l_f}$, so we have to verify that the source of $f$ is sent to $J^{k_f-1} q^{0}$. Now, the even subfaces of $f$ come in two flavors: there's the one which starts with $I^{k_f-1} p^0 p^{l_f}$, which $\kappa(n)$ sends to $J^{k_f-1} q^{0}$, and the other ones all start with $I^a p^c I^b p^{l_f}$, which forces $\kappa(n)$ to send them to some identity of a lower dimensional morphism. Therefore the composition of the even subfaces of $f$, i.e. the source of $f$, is sent to a composition of $J^{k_f-1} q^{0}$ with some identities of lower dimensional morphisms; but the only composition of that form available in $w(n)$ returns $J^{k_f-1} q^{0}$, which is what we needed. The same argument applied to $I^{k_f-1} p^1 p^{l_f}$ shows that the target of $f$ is sent to $J^{k_{f}-1} q^1$.
\end{proof}

We want to extend $\kappa(n)$ to a natural transformation of functors $\square^n \to w$. To do this we will need the following technical lemma:

\begin{lmm}
	Let $F, G : \bD^n \to \fCat_n$ be Segal functors valued in strict $n$-categories; in particular, the Segal maps are isomorphisms, not just equivalences. Let $\sigma, \tau : [0] \to [1]$ and $\rho : [1] \to [0]$ be the three maps in $\bD$ between $[0]$ and $[1]$, and let $\gamma : [1] \to [2]$ denote the unique active map. Let $\{ \alpha_{\undl{t}} : F(\undl{t}) \to G(\undl{t}) \mid \undl{t} \text{ is an elementary $n$-tuple}\}$ be a collection of maps defined for elementary $n$-tuples. Extend $\alpha_{\undl{t}} : F(\undl{t}) \to G(\undl{t})$ to non-elementary $n$-tuples via the Segal maps:
	\begin{equation*}
		\alpha_{\undl{t}} := \colim_{\undl{e} \to \undl{t}} \alpha_{\undl{e}} : F(\undl{t}) \cong \colim_{\undl{e} \to \undl{t}} F(\undl{e}) \to \colim_{\undl{e} \to \undl{t}}  G(\undl{e}) \cong G(\undl{t})
	\end{equation*}
	where the colimit is taken over all inert map $\undl{e} \to \undl{t}$ with $\undl{e}$ elementary. Then $\alpha_{\undl{t}}$ for $\undl{t} \in \bD^n$ form the components of a natural transformation $\alpha : F \to G$ if and only if the natural diagram 
	\begin{equation*}
		\begin{tikzcd}
			F(\undl{s}) \ar[r, "\alpha_{\undl{s}}"] \ar[d, "F(\vphi)"'] & G(\undl{s}) \ar[d, "G(\vphi)"] \\
			F(\undl{t}) \ar[r, "\alpha_{\undl{t}}"] & G(\undl{t})
		\end{tikzcd}
	\end{equation*}
	commutes for all $\undl{s}, \undl{t}, \vphi$ with the following properties:
	\begin{enumerate}
		\item $\undl{s} = (\undl{a}, s_i, \undl{b})$ and $\undl{t} = (\undl{a}, t_i, \undl{b})$ with $\undl{a}$ and $\undl{b}$ both elementary tuples;
		\item $\vphi_i : [s_i] \to [t_i]$ is one of $\sigma$, $\tau$, $\rho$, or $\gamma$;
		\item $\vphi_j = \id$ for all $j \neq i$.
	\end{enumerate}
\end{lmm}

\begin{proof}
	One direction is clear. For the other, we need to show that all naturality diagrams commute. It's enough to check this for face and degeneracy maps, as they generate $\bD^n$ under composition. A face map is the product of identities together with a map $[k] \to [k+1]$ whose image misses the value $i$; if $i = 0$ we can write this map as
	\begin{equation*}
		[k-1] \cong [0] \sqcup_{[0]} [k-1] \xlongrightarrow{\tau \sqcup \id} [1] \sqcup_{[0]} [k-1] \cong [k],
	\end{equation*} 
	if $i = k+1$ we can write it as
	\begin{equation*}
		[k-1] \cong [k-1] \sqcup_{[0]} [0] \xlongrightarrow{\id \sqcup \sigma} [k-1] \sqcup_{[0]} [1] \cong [k],
	\end{equation*}
	and if $1 \leq i \leq k$ we can write it as
	\begin{equation*}
		[k-1] \cong [i] \sqcup_{[0]} [1] \sqcup_{[0]} [k-i-2] \xlongrightarrow{\id \sqcup \gamma \sqcup \id} [i] \sqcup_{[0]} [2] \sqcup_{[0]} [k-i-2] \cong [k].
	\end{equation*} 
	Similarly, a degeneracy map is the product of identities together with a map of the form $[k+1] \to [k]$ with some $i$ in the codomain having exactly two preimages; we can always write the last map as
	\begin{equation*}
		[k] \cong [i] \sqcup_{[0]} [1] \sqcup_{[0]} [k-i-1] \xlongrightarrow{\id \sqcup \rho \sqcup \id} [i] \sqcup_{[0]} [0] \sqcup_{[0]} [k-i-1] \cong [k-1].
	\end{equation*}
	Thus we can observe that the naturality diagram of a face or degeneracy map is a colimit of naturality diagrams of maps which are products of identities and one of $\sigma, \tau, \rho, \gamma$. If the latter all commute, which is the assumption in the claim, then all their colimits commute; thus all naturality diagrams commute, concluding the proof.
\end{proof}

\begin{prp} \label{prp:kappa-existence}
	There is a natural transformation $\kappa : \square \to w$ of functors $\bD^n \to \fCat_{(\infty,n)}$ such that the component of $\kappa$ at $(1, \dotsc, 1)$ is given by $\kappa(n)$.
\end{prp}

\begin{proof}
	First we define $\kappa_{\undl{t}}$ on elementary $n$-tuples $\undl{t}$ as follows: if $\undl{t} = (\undl{1}, 0, \undl{s})$ with $\undl{s}$ an elementary $(n-k)$-tuple (so there are $k-1$ many $1$s at the start of $\undl{t}$), then
	\begin{equation*}
		\kappa_{\undl{t}} : \square(\undl{t}) \to \square(\undl{1}, \undl{0}) \simeq \square(k-1) \xlongrightarrow{\kappa(k-1)} w(k-1) \simeq w(\undl{1}, \undl{0}) \simeq w(\undl{t})
	\end{equation*}
	where the first map is induced by the projection $\undl{s} \to \undl{0}$, the second is the equivalence of \Cref{lmm:square-segal}, and the last two are the equivalences in \Cref{rmk:walking-morphisms}. Now we can define $\kappa_{\undl{t}}$ for non-elementary $\undl{t}$ by extending along the Segal decomposition of $\square(\undl{t})$ and $w(\undl{t})$. Using the lemma, to prove that $\kappa_{\undl{t}}$ is natural in $\undl{t}$ it will be enough to prove that the square
	\begin{equation*}
		\begin{tikzcd}
			\square^n(\undl{s}) \ar[r, "\kappa_{\undl{s}}"] \ar[d, "\square^n(\vphi)"'] & w(\undl{s}) \ar[d, "w(\vphi)"] \\
			\square^n(\undl{t}) \ar[r, "\kappa_{\undl{t}}"] & w(\undl{t})
		\end{tikzcd}
	\end{equation*}
	commutes for all $\undl{s}, \undl{t}, \vphi$ with the following properties:
	\begin{enumerate}
		\item $\undl{s} = (\undl{a}, s_i, \undl{b})$ and $\undl{t} = (\undl{a}, t_i, \undl{b})$ with $\undl{a}$ and $\undl{b}$ both elementary tuples;
		\item $\vphi_i : [s_i] \to [t_i]$ is one of $\sigma$, $\tau$, $\rho$, or $\gamma$;
		\item $\vphi_j = \id$ for all $j \neq i$.
	\end{enumerate}
	The cases $\vphi_i \in \{\sigma, \tau, \rho\}$ follow immediately from the definition of $\kappa_{\undl{t}}$. For $\vphi_i = \gamma$, $\undl{s} = (\undl{a}, 1, \undl{b})$, and $\undl{t} = (\undl{a}, 2, \undl{b})$ we have two cases
	\begin{enumerate}
		\item If $\undl{a}$ starts with $k$ many $1$s followed by at least one $0$ then we can rewrite the diagram as
		\begin{equation*}
			\begin{tikzcd}
				\square(\undl{a}, 1, \undl{b}) \ar[r] \ar[d] & \square(k) \ar[d, equal] \ar[r] & w(k) \ar[r, "\cong"] \ar[d, equal] & w(\undl{a}, 1, \undl{b}) \ar[d, "\cong"] \\
				\square(\undl{a}, 2, \undl{b}) \ar[r, "f"] & \square(k) \ar[r] & w(k) \ar[r, "\cong"] & w(\undl{a}, 2, \undl{b})
			\end{tikzcd}
		\end{equation*}
		where $f$ is the map
		\begin{equation*}
			\square(\undl{a}, 2, \undl{b}) \cong \square(\undl{a}, 1, \undl{b}) \sqcup_{\square(\undl{a}, 0, \undl{b})} \square(\undl{a}, 1, \undl{b}) \to \square(k) \sqcup_{\square(k)} \square(k) \simeq \square(k).
		\end{equation*}
		This ends up being equivalent to the projection $\square(\undl{a}, 2, \undl{b}) \to \square(k)$ induced by the map $(\undl{a}, 2, \undl{b}) \to (\undl{1}, 0, \dotsc, 0)$, so the left square commutes. The other two obviously commute as well, and so the original diagram commutes.
		\item If $\undl{a} = \undl{1}$ then we may as well assume $\undl{b} = \undl{1}$, since the case where $\undl{b}$ contains a $0$ can be turned into a lower dimensional case where $\undl{b}$ does not contain any $0$s. Assume the $2$ appears in coordinate $k + 1$. Then the diagram reads
		\begin{equation*}
			\begin{tikzcd}
				\square(n) \ar[rrr, "\kappa(n)"] \ar[d] & & & w(n) \ar[d] \\
				\square(n) \sqcup_{\square(\undl{1}, 0, \undl{1})} \square(n) \ar[r] & \square(n) \sqcup_{\square(k)} \square(n) \ar[rr, "\kappa(n) \sqcup_{\kappa(k)} \kappa(n)"'] & & w(n) \sqcup_{w(k)} w(n)
			\end{tikzcd}
		\end{equation*}
		That this diagram commutes can be verified on the generators of $\square(n)$ and it's mainly just an exercise in understanding the notation, so we leave it to the reader. \qedhere
	\end{enumerate}
\end{proof}

\begin{exm}
	To see how $\kappa$ is defined on elementary representables in low dimensional cases we refer the reader to \Cref{exm:kappa}. There, the left red face is sent to the left red vertex of $w(3)$, and the map doing that is precisely $\kappa_{0,1,1} : \kappa(0,1,1) \to w(0,1,1) \cong w(0)$; similarly for the right olive face. The front face is collapsed to an edge via the map $\kappa_{1,0,1}$, and similarly for the back face, while the top and bottom faces are sent to copies of $w(2)$ via $\kappa_{1,1,0}$.
\end{exm}

\subsection{Higher square functor}

\begin{dfn}
	The \emph{higher square functor} $\Sq^n : \Glb^n(\Spaces) \to \Seg^n(\Spaces)$ is defined by currying the functor
	\begin{equation*}
		\bD^{n, \op} \times \GlbSeg^{n,n}(\Spaces) \xlongrightarrow{(\square^n)^\op \times \id} \fCat_{(\infty,n)}^\op \times \GlbSeg^{n,n}(\Spaces) \sub \GlbSeg^{n,n}(\Spaces)^\op \times \GlbSeg^{n,n}(\Spaces) \xlongrightarrow{\Map(-,-)} \Spaces
	\end{equation*}
	to a functor $\GlbSeg^{n,n}(\Spaces) \to \Fun(\bD^{n,\op}, \Spaces)$, which ends up landing in $\Seg^n(\Spaces)$.
\end{dfn}

\begin{rmk}
	For any $X \in \Glb^n(\Spaces)$ the $\undl{t}$-morphisms of $\Sq^n(X)$ satisfy
	\begin{equation*}
		\Sq^n(X)_{\undl{t}} \simeq \Map(\square^n(\undl{t}), X).
	\end{equation*}
	Thus, given the description of \Cref{prp:cube-structure}, we can think of points of $\Sq^n(X)_{\undl{t}}$ as certain iterated cubical diagrams in $X$.
\end{rmk}

\begin{exm} \label{exm:sq2}
	In the case $n = 2$ the $(a,b)$-morphisms of $\Sq^2(X)$ are easy to draw in terms of the morphisms of $X$: they are diagrams in $X$ of the form
	\begin{equation*}
		\begin{tikzcd}[sstyle]
			\bullet \ar[r] \ar[d] & \bullet \ar[r] \ar[d] \ar[dl, phantom, "\rotatebox{-45}{$\Downarrow$}" description] & \bullet \ar[r] \ar[d] \ar[dl, phantom, "\rotatebox{-45}{$\Downarrow$}" description] & \bullet \ar[d] \ar[dl, phantom, "\rotatebox{-45}{$\Downarrow$}" description] \\
			\bullet \ar[r] \ar[d] & \bullet \ar[r] \ar[d] \ar[dl, phantom, "\rotatebox{-45}{$\Downarrow$}" description] & \bullet \ar[r] \ar[d] \ar[dl, phantom, "\rotatebox{-45}{$\Downarrow$}" description] & \bullet \ar[d] \ar[dl, phantom, "\rotatebox{-45}{$\Downarrow$}" description] \\
			\bullet \ar[r] & \bullet \ar[r] & \bullet \ar[r] & \bullet \\
		\end{tikzcd}
	\end{equation*}
	with $a$ columns and $b$ rows of $2$-morphsims $\Rightarrow$, each one going between two composite $1$-morphisms. The composition of two diagrams is given by pasting them together and then composing the resulting $2$-morphisms. 
\end{exm}

\begin{rmk}
	Recall that there is a natural transformation $\kappa : \square^n \to w$ (\Cref{prp:kappa-existence}). From the definition of $\Sq^n$ we see that $\kappa$ induces a natural transformation of functors from the inclusion $\GlbSeg^{n,n}(\Spaces) \sub \Seg^n(\Spaces)$ to the higher square functor $\Sq^n : \GlbSeg^{n,n}(\Spaces) \to \Seg^n(\Spaces)$, which we denote by $\kappa^\ast$. In particular we have a map $\kappa^\ast_X : X \to \Sq^n(X)$ for any $X \in \GlbSeg^{n,n}(\Spaces)$ which we call the \emph{canonical inclusion}. Indeed, it can be shown that $\kappa_X^\ast$ induces an equivalence $X \simeq U^n \Sq^n(X)$ between $X$ and the underlying fully globular $n$-uple Segal space of $\Sq^n(X)$.
\end{rmk}

\begin{exm}
	In the case $n = 2$, by examining \Cref{exm:kappa} we see that the canonical inclusion sends a $2$-morphism $\alpha : w(2) \to X$ of $X$ to the $(1,1)$-morphism $\kappa^2_X(\alpha) \simeq \alpha \circ \kappa(2) : \square^2 \to w(2) \to X$ of $\Sq^2(X)$ obtained by thickening $\alpha$ in the vertical direction:
	
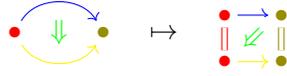
\begin{equation*}
		\begin{tikzcd}[sstyle]
			{\color{red} \bullet} \ar[rr, bend left = 60, color = blue] \ar[rr, bend right = 60, color = yellow] \ar[rr, phantom, "\Downarrow" description, color = green] & & {\color{olive} \bullet}
		\end{tikzcd}
		\quad \mapsto \quad
		\begin{tikzcd}[sstyle]
			{\color{red} \bullet} \ar[r, color = blue] \ar[d, color = red, equal] & {\color{olive} \bullet} \ar[d, color = olive, equal] \ar[dl, phantom, "\rotatebox{-45}{$\Downarrow$}" description, color = green] \\
			{\color{red} \bullet} \ar[r, color = yellow] & {\color{olive} \bullet}
		\end{tikzcd}
	\end{equation*}
	For arbitrary $n$ the picture is very similar: an $n$-morphism is taken to a thickened version of itself that fits in a cubical mold.
\end{exm}

Next we introduce \emph{signatures} for the square functor. 
\begin{dfn}
	Consider the walking span
	\begin{equation*}
		\Lambda : \, \quad - \leftarrow 0 \rightarrow +
	\end{equation*}
	and its iterated product $\Lambda^n$. An object of $\undl{b} \in \Lambda^n$, called a \emph{signature}, is an $n$-tuple valued in $\{-,0,+\}$. For each symbol $s \in \{-, 0, +\}$ we define a functor $F_s : \bD \to \bD$: 
	\begin{itemize}
		\item $F_+$ is the identity functor;
		\item $F_{-}$ is the functor reversing the linear order of each element of $\bD$;
		\item $F_0$ is the constant functor on $[0]$.
	\end{itemize}
	Thus for every signature $\undl{b} \in \Lambda^n$ we get a functor $F_{\undl{b}} = \prod_i F_{b_i} : \bD^n \to \bD^n$.
\end{dfn}

\begin{dfn}
	The \emph{higher square functor with signature $\undl{b} \in \Lambda^n$}, denoted $\Sq^n_{\undl{b}} : \GlbSeg^{n,n}(\Spaces) \to \Seg^n(\Spaces)$, is defined by currying the functor
	\begin{equation*}
		\bD^{n, \op} \times \GlbSeg^{n,n}(\Spaces) \xlongrightarrow{((\square^n)^\op \circ F_{\undl{b}}^\op) \times \id} \fCat_{(\infty,n)}^\op \times \GlbSeg^{n,n}(\Spaces) \sub \GlbSeg^{n,n}(\Spaces)^\op \times \GlbSeg^{n,n}(\Spaces) \xlongrightarrow{\Map(-,-)} \Spaces
	\end{equation*}
	to a functor $\GlbSeg^{n,n}(\Spaces) \to \Fun(\bD^{n,\op}, \Spaces)$ which ends up landing in $\Seg^n(\Spaces)$.
\end{dfn}

\begin{rmk} \label{rmk:explicit-sq}
	Explicitly, say that $\undl{b}$ has a $0$ in coordinates $z_1 < z_2 < \dotsb < z_p$ and a $-$ in coordinates $m_1 < m_2 < \dotsc < m_q$; then
	\begin{equation*}
		\Sq^n_{\undl{b}}(X) \simeq I_{z_p} \dotsm I_{z_2} I_{z_1} ( \Sq^{n-p}(\tau_{n-p} X)^{\op_{m_1}, \op_{m_2}, \dotsc, \op_{m_q}} )
	\end{equation*}
	where $I_r : \Seg^k(\Spaces) \to \Seg^{k+1}$ is the left adjoint of the evaluation functor at $[0]$ on the $r$th factor of $\bD$. This follows from the fact that $\square^n(\undl{t},\undl{0}) \simeq \square^{n-p}(\undl{t})$ for any $(n-p)$-tuple $\undl{t}$. In particular, for any $\undl{c} = (\undl{a}, \undl{b}) \in \Lambda^{m} \times \Lambda^n$ we have
	\begin{equation*}
		\tau_m \Sq^{m+n}_{\undl{c}}(X) \simeq \Sq^{m+n}_{\undl{a},\undl{0}}(X) \simeq \Sq^{m}_{\undl{a}}(\tau_m X).
	\end{equation*}
\end{rmk}

\begin{rmk}
	The assignment $\undl{b} \mapsto \Sq^n_{\undl{b}}$ is functorial over $\Lambda^n$: there are maps
	\begin{equation*}
		\Sq^n_{b_1, \dotsc, 0, \dotsc, b_n}(X) \to \Sq^n_{b_1, \dotsc, -, \dotsc, b_n}(X), \quad \Sq^n_{b_1, \dotsc, 0, \dotsc, b_n}(X) \to \Sq^n_{b_1, \dotsc, +, \dotsc, b_n}(X),
	\end{equation*}
	where we changed the $i$th coordinate, and these maps pairwise commute for different values of $i$. 
\end{rmk}

\subsection{Zigzagification}

Fix $X \in \GlbSeg^{n,n}(\Spaces)$. Since in particular $X \in \GlbSeg^{n+1,n+1}(\Spaces)$, for every $\undl{b} \in \Lambda^n$ we can construct the $(n+1)$-uple Segal space $\Sq^{n+1}_{\undl{b},+}(X)$, naturally in $\undl{b}$ and $X$. Let $Z^{n+1}_+(X) \in \Fun(\bD^{n+1,\op}, \Spaces)$ denote the presheaf obtained by taking the pointwise colimit of $\Sq^{n+1}_{\undl{b},+}(X)$:
\begin{equation*}
	Z^{n+1}_+(X)_{\undl{t}} := \colim_{\undl{b} \in \Lambda^n} \left( \Sq^{n+1}_{\undl{b},+} (X)_{\undl{t}} \right).
\end{equation*}

Note that $Z^{n+1}_+(X)$ is \emph{not} an $(n+1)$-uple Segal space, since the colimit was taken only pointwise. At this point we could apply the Segalification functor $\LL^{n+1}$ from \Cref{sec:segalification} but doing so naively would not yield the desired result. Instead, consider the following:

\begin{dfn} \label{dfn:zigzag}
	Let $X \in \GlbSeg^n(\Spaces)$. The \emph{(positive) zigzagification of $X$} is
	\begin{equation*}
		\zig_{+}^{n+1}(X) := (R_{n+1}^{n+1} \LL_{n+1} R_{n}^{n+1} \LL_{n} \dotsb R_{1}^{n+1} \LL_{1}) (Z^{n+1}_+(X)),
	\end{equation*}
	which is a globular $(n+1)$-uple Segal space by \Cref{crl:alternate-segalification}.
\end{dfn}

\begin{rmk}
	The canonical inclusion $\kappa_X^\ast : X \to \Sq^{n+1}(X)$ extends to a map $X \to \zig_{+}^{n+1}(X)$ of globular $n$-uple Segal spaces: this is because $X$ is already globular and Segal, and so applying $R_i^n$ and $\LL_j$ does not affect it.
\end{rmk}

\begin{rmk}
	There is a \emph{negative} zigzagification $\zig_{-}^{n+1}(X)$ of $X$ as well, obtained by replacing the $+$ with a $-$ in the definition of $Z^{n+1}_+(X)$. It can be readily seen that $Z^{n+1}_{-}(X) \simeq (Z^{n+1}_+(X))^{\op_{n+1}}$ and so $\zig_-^{n+1}(X) \simeq \zig_{+}^{n+1}(X)^{\op_{n+1}}$ since segalification and globularization commute with taking opposites.
\end{rmk}

\begin{prp} \label{prp:truncation-zigzag}
	For any $0 \leq k \leq n$ we have
	\begin{equation*}
		\tau_k \zig_{+}^{n+1}(X) \simeq (R_k^k \LL_k \dotsb R_1^k \LL_1)(\tau_k Z^{n+1}_+(X)) \simeq (R_k^k \LL_k \dotsb R_1^k \LL_1)(Z^{k}(\tau_k X))
	\end{equation*}
	where $Z^{k}(Y)_{\undl{t}} \simeq \displaystyle \colim_{\undl{b} \in \Lambda^k} \left( \Sq^{k}_{\undl{b}} (Y)_{\undl{t}} \right)$ for any $Y \in \GlbSeg^{k,k}(\Spaces)$.
\end{prp}

\begin{proof}
	The first equivalence is a consequence of the fact that $\LL_j$ and $R_j^m$ commute with evaluation at $[0]$ for any $j$ and $m$. For the second equivalence, write
	\begin{equation*}
		Z_+^{n+1}(X)_{\undl{t}} \simeq \colim_{\undl{c} \in \Lambda^{n-k-1}} \colim_{\undl{b} \in \Lambda^k} \left( \Sq^{n+1}_{\undl{b},\undl{c},+}(X) \right)_{\undl{t}}.
	\end{equation*}
	Then applying $\tau_k$ turns each $\Sq^{n+1}_{\undl{b},\undl{c},+}(X)$ into $\Sq^{k}_{\undl{b},\undl{0}}(X) \simeq \Sq^k_{\undl{b}}(\tau_k X)$ by \Cref{rmk:explicit-sq} and the outermost colimit becomes a constant colimit at $\Sq^k_{\undl{b}}(\tau_k X)$.
\end{proof}

Using the formula for $\LL_k$ given in \Cref{fct:formula-segalification} and the previous proposition we can describe $\zig_{+}^{n+1}(X)$ informally as follows:
\begin{itemize}
	\item its objects are the objects of $X$;
	\item its $1$-morphisms are zigzags $1$-morphisms of $X$;
	\item its $2$-morphisms are obtained by taking all lax commutative squares in $X$, forming globular zigzags in the first direction, and then forming zigzags of those in the second direction;
	\item its $3$-morphisms are obtained by taking all lax commuative cubes in $X$, forming globular zigzags in the first direction, then forming globular zigzags of those in the second direction, and finally forming zigzags of those in the third direction;
	\item its $k$-morphisms, in general, are iterated zigzags of lax commutative $k$-cubes in $X$ where the zigzags are taken in order from the first direction to the $k$th direction.
\end{itemize}

\section{Free adjoints and zigzagification} \label{sec:properties}

In this section we will prove that the functor $\zig_{+}^2 : \GlbSeg^{1,1}(\Spaces) \to \GlbSeg^{2,2}(\Spaces)$ freely adds right adjoints to the morphisms of a Segal space. The main technical hurdle is to show that, in a technical sense, $\zig_+^2(X)$ is generated by the adjunction data. We first prove some helpful results about generating spaces in globular Segal spaces.

\subsection{Some results on generators}

\begin{dfn}
    Let $X \in \GlbSeg^{n,n}(\Spaces)$ be a globular $n$-uple Segal space. Let $A \hookrightarrow X_{\undl{1^n}}$ be a monomorphism. Recall that the Segal condition gives us composition operations $- \circ_j -$ on $X_{\undl{1^n}}$ for $1 \leq j \leq n$. We say that $A$ \emph{generates $X$ under composition} if it has the following properties:
    \begin{enumerate}
        \item for all $0 \leq k \leq n-1$, the iterated identity map $\id(-) : X_{\undl{1^k}, \undl{0^{n-k}}} \to X_{\undl{1^n}}$ factors through $A$;
        \item the smallest subspace of $X_{\undl{1^n}}$ which contains $A$ and is closed under all the composition operations is $X_{\undl{1^n}}$.
    \end{enumerate}
    We call the $k$-morphisms $f \in A$ the \emph{generating $k$-morphisms}.
\end{dfn}

\begin{rmk} \label{rmk:extension-to-el}
    The first condition implies that $A$ is closed under the source and target maps, as those land in the morphisms of dimension strictly lower than $n$. In particular, if $\bD^{n,\op}_{\el}$ denotes the full subcategory of $\bD^{n,\op}$ spanned by the elementary $n$-tuples (recall that these are the ones whose entries are in $\{0,1\}$), then $A$ determines a functor $A^\el : \bD^{n,\op}_{\el} \to \Spaces$ with $A^\el_{\undl{1^n}} \simeq A$ and $A^{\el}_{\undl{e}} \simeq X_{\undl{e}}$ for any other elementary $\undl{e} \neq \undl{1^n}$.
\end{rmk}

\begin{rmk} \label{rmk:extension-to-simp}
    Let $X \in \GlbSeg^{n,n}(\Spaces)$. Any functor $W : \bD^{n,\op}_{\el} \to \Spaces$ with a map to $W \to X\vert_{\bD^{n,\op}_{\el}}$ can be extended to a functor $M(W) : \bD^{n,\op} \to \Spaces$ with a map $M(W) \to X$ that returns the given map upon restricting to $\bD^{n,\op}_{\el}$. Indeed, for any $[\undl{t}] \in \bD^{n, \op}$ we can define
    \begin{equation*}
        M(W)_{\undl{t}} := \left( \lim_{[\undl{e}] \to [\undl{t}]} W_{\undl{e}} \right) \times_{\left( \displaystyle \lim_{[\undl{e}] \to [\undl{t}]} X_{\undl{e}} \right)} X_{\undl{t}} \hookrightarrow X_{\undl{t}}
    \end{equation*}
    where the limits are taken over elementary $n$-tuples $\undl{e}$; it's not hard to see that $M(W)_{\undl{t}}$ is natural in $\undl{t}$ and that the projections to $X_{\undl{t}}$ define the induced map $M(W) \hookrightarrow X$. In the case $W = A^{\el}$, $M(A^{\el})$ is the $n$-simplicial space whose value at $\undl{t}$ is the subspace of $X_{\undl{t}}$ all whose elementary components (the projections to $X_{\undl{e}}$) land in $A$.
\end{rmk}

\begin{prp}
    Let $X \in \GlbSeg^{n,n}(\Spaces)$ and let $A \hookrightarrow X_{\undl{1^n}}$ generate $X$ under composition. Then there is a sequence of $n$-simplicial spaces
    \begin{equation*}
        M(A^{\el}) = M(A(1)^{\el}) \hookrightarrow M(A(2)^{\el}) \hookrightarrow M(A(3)^{\el}) \hookrightarrow \dotsb
    \end{equation*}
    whose colimit is equivalent to $X$.
\end{prp}

\begin{proof}
    Fix a $p \geq 1$ and let $I_p$ denote the category of inert maps $[\undl{e}] \to [\undl{p^n}]$ in $\bD^n$ such that $\undl{e}$ is an elementary $n$-tuple. Define the space
    \begin{equation*}
        A(p) := \mathrm{im} \left( \lim_{([\undl{e}] \to [\undl{p^n}]) \in I_p} A^{\el}_{\undl{e}} \to \lim_{([\undl{e}] \to [\undl{p^n}]) \in I_p} X_{\undl{e}} \xlongrightarrow{\simeq} X_{\undl{p^n}} \xlongrightarrow{c} X_{\undl{1^n}} \right) \hookrightarrow X_{\undl{1^n}} 
    \end{equation*}
    as the essential image of the total composition map $c : X_{\undl{p^n}} \to X_{\undl{1^n}}$ (induced by the function $[1] \cong \{0,p\} \to [p]$ in all coordinates) with its domain restricted so that the elementary components of each $f \in X_{\undl{p^n}}$ come from $A$. Informally (but correctly), we think of $A(p)$ as the subspace of $X_{\undl{1^n}}$ containing those morphisms which can be obtained starting from $A$ by doing at most $p$ iterated compositions in each direction (so a total of $n p$ compositions). Note that since $A$ generates $X$ and $A \sub A(p)$ for any $p$, $A(p)$ also generates $X$. Using \Cref{rmk:extension-to-el} and \Cref{rmk:extension-to-simp} we can extend $A(p)$ to $M(A(p)^{\el})$. It's clear that $A(1) \simeq A$ and hence that $M(A(1)^{\el}) \simeq M(A^{\el})$, and the induced maps come from the inclusions $A(p) \hookrightarrow A(p+1)$ obtained, for example, by composing with identities at the last step in every direction. 
    
    Finally we need $X \simeq \displaystyle \colim_{p \to \infty} M(A(p)^{\el})$. This is a filtered colimit, so it will be enough to show that every $\undl{t}$-morphism $f$ of $X$ appears at some point in $M(A(p)^{\el})_{\undl{t}}$ for some $p$. Indeed, by the Segal condition, it is sufficient for the elementary components $\{f_a\}_{a = 1, \dotsc, b}$ of $f$ (the order is irrelevant, the important thing is that all of them are listed) to appear at some stage; since each of those appears at some stage $p_a$ by the assumption that $A$ generates $X$, we can take $p = \max \{p_a\}$.
\end{proof}

\begin{prp}
    Let $X, Y \in \GlbSeg^{n,n}(\Spaces)$ and let $A \hookrightarrow X_{\undl{1^n}}$ generate $X$ under composition. Then the canonical map
    \begin{equation*}
        \Map(X, Y) \to \Map(M(A(\undl{1^n})^{\el}), Y)
    \end{equation*}
    is a monomorphism of spaces.
\end{prp}

\begin{proof}
    A map $F : X \to Y$ is uniquely determined by its compatible components $F_{\undl{t}} : X_{\undl{t}} \to Y_{\undl{t}}$ (see, for example, the end formula for natural transformations in \cite[Proposition 5.1]{GHN2017}). We will show that each of those components is uniquely determined by its restriction to $M(A(1)^{\el})_{\undl{t}}$. For notational simplicity we will fix $n$ to be $1$ (so that the tuple $\undl{t}$ becomes just $t$), trusting that the reader won't have any troubles extending the argument to an arbitrary $n$. 

    By naturality and the Segal condition, the map $F_t : X_t \to Y_t$ is uniquely determined by the maps $F_1$ and $F_0$; the latter is also determined by $F_1$ via the map $\id(-) : X_0 \to X_1$, so we will focus on the former. Since $X_1 \simeq \displaystyle \colim_{p \to \infty} A(p)$ we have that $F_1$ is determined by $F_1(p) : A(p) \to Y_1$. Consider now the composition map 
    \begin{equation*}
        \underbrace{A(1) \times_{X_0} \dotsb \times_{X_0} A(1)}_{\text{$p$ copies}} \to X_1 \times_{X_0} \dotsb \times_{X_0} X_1 \simeq X_{p} \to X_1
    \end{equation*}
    whose essential image is $A(p)$. Postcomposing with $F_1$ and using the Segal condition on $Y$ gives a commutative square
    \begin{equation*}
        \begin{tikzcd}
            \underbrace{A(1) \times_{X_0} \dotsb \times_{X_0} A(1)}_{\text{$p$ copies}} \ar[r, "F_1(1)^{\times p}"] \ar[d] & \underbrace{Y_1 \times_{Y_0} \dotsb \times_{Y_0} Y_1}_{\text{$p$ copies}} \ar[d] \\
            A(p) \ar[r, "F_1(p)"'] & Y_1
        \end{tikzcd}
    \end{equation*}
    which implies that $F_1(p)$ is uniquely determined by $F_1(1)$. Combining these observations, we deduce that $F : X \to Y$ is uniquely determined by $F\vert_{M(A(1)^\el)} : M(A(1)^\el) \to Y$, as needed.
\end{proof}

\subsection{Adjoints from zigzags of commutative squares}

Fix a Segal space $X$. We want to apply the above results to the case of $\zig^2_+(X)$ so that we can simplify the space of maps out of $\zig^2_+(X)$. Before tackling $\zig_{+}^2(X)$ directly, let's start by understanding $\Sq^2(X)$. Any $1$-morphism $f : x \to y$ of $X$ gives rise to two special $(1,1)$-morphisms of $\Sq^2(X)$, depicted as commutative squares $\square^2 \to X$ in $X$:
\begin{equation*}
	E_f := \,
	\begin{tikzcd}
		x \ar[r, "f"] \ar[d, "f"'] & y \ar[d, equal] \\
		y \ar[r, equal] & y
	\end{tikzcd}
	\qquad \qquad
	H_f := \,
	\begin{tikzcd}
		x \ar[r, equal] \ar[d, equal] & x \ar[d, "f"] \\
		x \ar[r, "g"'] & y
	\end{tikzcd}
\end{equation*}
By pasting the squares together we can see that they satisfy
\begin{equation*}
	E_f \circ_1 H_f \simeq \id_2(f^\to) \qquad \text{and} \qquad E_f \circ_2 H_f \simeq \id_1(f^\downarrow)
\end{equation*}
where $f^\to$ and $f^\downarrow$ are, respectively, the $(1,0)$-morphism and the $(0,1)$-morphism in $\Sq^2(X)$ induced by the $1$-morphism $f$ of $X$. In this situation say that $f^\to$ and $f^\downarrow$ are \emph{companions}, which is a notion that can be defined for arbitrary double Segal spaces (\Cref{dfn:companions}). In fact, $\Sq^2(X)$ is universal with respect to those double Segal spaces that admit certain companions. We will discuss this in more detail in \Cref{sec:extensions}.

\begin{rmk} \label{rmk:E-H-decomp}
	Part of proving the universality of $\Sq^2(X)$ involves showing that the $(1,1)$-morphisms of the form $E_f$, $H_f$, and $\id_2(f)$ generate all $(1,1)$-morphisms under composition. This follows from a straightforward calculation: for any commutative square determined by $g \circ_1 f \simeq k \circ_1 h$ in $X$ we have
	\begin{equation*}
		\begin{tikzcd}
			w \ar[r, "f"] \ar[d, "h"'] & x \ar[d, "g"] \\
			y \ar[r, "k"'] & z
		\end{tikzcd}
		\quad \simeq \quad
		\begin{tikzcd}
			w \ar[d, equal] \ar[r, "f"] & y \ar[d, equal] \ar[dr, phantom, "\circ_1"] & y \ar[d, equal] \ar[r, equal] & y \ar[d, "g"] \\
			w \ar[r, "f"] & y & y \ar[r, "g"] & z \ar[ddlll, phantom, "\circ_2"] \\ \\
			w \ar[r, "h"'] \ar[d, "h"'] & x \ar[d, equal] \ar[dr, phantom, "\circ_1"] & y \ar[r, "k"'] \ar[d, equal] & z \ar[d, equal] \\
			y \ar[r, equal] & y & y \ar[r, "k"'] & z
		\end{tikzcd}
		\quad \simeq \quad
		(\id_2(k^\to) \circ_1 E_h) \circ_2 (H_g \circ_1 \id_2(f^\to)).
	\end{equation*}
\end{rmk}

Analogous results hold in $\Sq^2_{-+}(X)$: for every $1$-morphism $f^{\op_1}$ of $X^{\op_1}$ (where we use the superscript $(-)^{\op_1}$ to denote that $f$ is the original $1$-morphism in $X$) there are $2$-morphisms $E_{f^{\op_1}}$ and $H_{f^{\op_1}}$ satisfying $H_{f^{\op_1}} \circ_1 E_{f^{\op_1}} \simeq \id_2((f^{\op_1})^\to)$ and $E_{f_{\op_1}} \circ_2 H_{f^{\op_1}} \simeq \id_{1}((f^{\op_1})^\downarrow)$, and these $(1,1)$-morphisms plus the identities generate the whole space of $(1,1)$-morphisms under composition. For notational convenience we will use $f^\leftarrow$ to denote $(f^{\op_1})^\to$. 

\begin{prp} \label{prp:adj-in-zig}
	For any $1$-morphism $f : x \to y$ in $X$ there is an adjunction $f^\to \dashv f^\leftarrow$ in $\zig_{+}^2(X)$ with unit $\eta = H_{f^{\op_1}} \circ_1 H_{f}$ and counit $\varepsilon = E_f \circ_1 E_{f^{\op_1}}$.
\end{prp}

\begin{proof}
	The first step is to ensure that $f^\to$, $f^\leftarrow$, $\eta$, and $\varepsilon$ are valid morphisms in $\zig_{+}^2(X)$. Recall that
	\begin{equation*}
		\zig_{+}^2(X) := (R^2_2 \LL_2 R_1^2 \LL_1) \left( Z^{2}_+(X) \right), \qquad Z^2_+(X)_{\undl{t}} := \Sq^2_{++}(X)_{\undl{t}} \sqcup_{\Sq^2_{0+}(X)_{\undl{t}}} \Sq^2_{-+}(X)_{\undl{t}}.
	\end{equation*}
	The space of $1$-morphisms of $\zig_{+}^2(X)$ contains the space $Z^2_+(X)_{1,0}$ and so, in particular, $f^\to \in \Sq^2_{++}(X)_{1,0}$ and $f^\leftarrow \in \Sq^2_{-+}(X)_{1,0}$ are valid $1$-morphisms. The space of $(1,1)$-morphisms of $\LL_1 Z^2_+(X)$ contains $Z^2_+(X)_{1,1}$, hence it contains $H_f, E_f \in \Sq^2_{++}(X)_{1,1}$ and $H_{f^{\op_1}}, E_{f^{\op_1}} \in \Sq^2_{-+}(X)_{1,1}$ and therefore it contains the (free) compositions $\eta$ and $\varepsilon$. Note also that $\eta$ and $\varepsilon$ are already globular: if $\partial_1^0 f \simeq x$ and $\partial_1^1 f \simeq y$ then
	\begin{equation*}
		\partial_1^0 \eta \simeq \partial_1^0 H_f \simeq \id^2(x) \simeq \partial_1^1 H_{f^{\op_1}} \simeq \partial_1^1 \eta, \qquad \partial_1^0 \varepsilon \simeq \partial_1^0 E_{f^{\op_1}} \simeq \id^2(y) \simeq \partial_1^1 E_{f} \simeq \partial_1^1 \varepsilon.
	\end{equation*}
	Therefore $\eta, \varepsilon \in (R_1^2 \LL_1 Z^2_+ (X))_{1,1}$ and so $\eta, \varepsilon \in \zig_{+}^2(X)_{1,1}$.
	
	Now we have to show that the snake equations hold. The argument is the same for both equations, so we will only write down one of them. More precisely, we will prove that there is some $\alpha \in (\LL_1 Z^2_+(X))_{1,2}$ whose components are
	\begin{equation*}
		\alpha_{01} \simeq \id_2(f^\to) \circ_1 \eta, \quad \alpha_{12} \simeq \varepsilon \circ_1 \id_2(f^\to), \quad \alpha_{02} \simeq \id_2(f^\to),
	\end{equation*}
	where $\alpha_{ij} \in (\LL_1 Z^2_+(X))_{1,1}$ denotes the restriction of $\alpha$ along the map $\{i,j\} \to [2]$ in the second coordinate; moreover, each $\alpha_{ij}$ is globular (as can be seen by computing the vertical boundaries) and so $\alpha$ belongs to $(R_1^2 \LL_1 Z^2_+(X))_{1,2}$. The existence of this $\alpha$ is enough to prove the snake equations because, in $\zig_{+}^2(X)$, $\alpha$ is turned into a witness for an equivalence $\alpha_{12} \circ_2 \alpha_{01} \simeq \alpha_{02}$.
	
	Consider the following points of $Z^2_+(X)_{1,2}$:
	\begin{align*}
		\beta_1 & := (H_f, \id_2(f^\to)) \in \Sq^2_{++}(X)_{1,1} \times_{\Sq^2_{++}(X)_{1,0}} \Sq^2_{++}(X)_{1,1} \simeq \Sq^2_{++}(X)_{1,2}, \\
		\beta_2 & := (H_{f^{\op_1}}, E_{f^{\op_1}}) \in \Sq^2_{-+}(X)_{1,1} \times_{\Sq^2_{-+}(X)_{1,0}} \Sq^2_{-+}(X)_{1,1} \simeq \Sq^2_{-+}(X)_{1,2}, \\
		\beta_3 & := (\id_{2}(f^\rightarrow), E_f) \in \Sq^2_{++}(X)_{1,1} \times_{\Sq^2_{++}(X)_{1,0}} \Sq^2_{++}(X)_{1,1} \simeq \Sq^2_{++}(X)_{1,2},
	\end{align*}
	where we have encoded a point $\gamma$ of $\Sq^{2}_{\pm, +}(X)_{1,2}$ as the pair $(\gamma_{01}, \gamma_{12})$ consisting of its two components in $\Sq^{2}_{\pm, +}(X)_{1,1}$ (we can do this since $\Sq^2(-)$ is a double Segal space). Note that we have
	\begin{equation*}
		\partial_1^1 \beta_1 \simeq (f^\downarrow, \id^2(y)) \simeq \partial_1^0 \beta_2, \qquad \partial_1^1 \beta_2 \simeq (\id^2(x), f^\downarrow) \simeq \partial_1^0 \beta_3
	\end{equation*}
	in $\Sq^{2}_{0+}(X)_{0,2}$ and so we can form the horizontal composition $\alpha := \beta_3 \circ_1 \beta_2 \circ_1 \beta_1$ in $\LL_1 Z^2_+(X)$. Since $\alpha \in \Sq^2_{++}(X)_{1,2} \times_{\Sq^2_{++}(X)_{0,2}} \Sq^2_{++}(X)_{1,2} \times_{\Sq^2_{++}(X)_{0,2}} \Sq^2_{++}(X)_{1,2}$, the formula for $\LL_1$ in \Cref{fct:formula-segalification} tells us that we can compute the components of $\alpha$ as the free composition of the components of the $\beta_i$. Therefore
	\begin{align*}
		\alpha_{01} & \simeq (\beta_3)_{01} \circ_1 (\beta_2)_{01} \circ_1 (\beta_1)_{01} \simeq \id_2(f^\to) \circ_1 H_{f^{\op_1}} \circ_1 H_f \simeq \id_2(f^\to) \circ_1 \eta, \\
		\alpha_{12} & \simeq (\beta_3)_{12} \circ_1 (\beta_2)_{12} \circ_1 (\beta_1)_{12} \simeq E_f \circ_1 E_{f^{\op_1}} \circ_1 \id_2(f^\to) \simeq \varepsilon \circ_1 \id_2(f^\to), \\
		\alpha_{02} & \simeq (\beta_3)_{02} \circ_1 (\beta_2)_{02} \circ_1 (\beta_1)_{02} \simeq E_f \circ_1 \id_1(f^\downarrow) \circ_1 H_f \simeq \id_2(f^\to),
	\end{align*}
	which is exactly what we needed.
\end{proof}

\begin{rmk}
	Here is a short diagrammatic explanation of the proof of \Cref{prp:adj-in-zig} and why it's not as straightforward as it could be. Write the diagram
	\begin{equation*}
		\begin{tikzcd}
			x \ar[r, equal] \ar[d, equal] & x \ar[d, "f"] \ar[r, equal] & x \ar[d, equal] \ar[r, "f"] & y \ar[d, equal] \\
			x \ar[r, "f"] \ar[d, equal] & y \ar[d, equal] & x \ar[l, "f"'] \ar[d, "f"] \ar[r, "f"] & y \ar[d, equal] \\
			x \ar[r, "f"] & y \ar[r, equal] & y \ar[r, equal] & y
		\end{tikzcd}
	\end{equation*}
	representing a formal vertical composite of two $2$-morphisms in $\zig_{+}^2(X)$, with the top row representing $\id_2(f^\to) \circ_1 \eta$ and the bottom row representing $\varepsilon \circ_1 \id_2(f^\to)$. If we didn't have to worry about globularity then we could simply use the interchange law for double Segal spaces, which holds in $\zig_{+}^2(X)$, to compute the vertical composition of each column first and then compose the resulting $(1,1)$-morphisms horizontally. In our setting, however, we can only perform vertical compositions of \emph{globular} $(1,1)$-morphisms (as the $\LL_2$ in $\zig_{+}^2(X) = (R_2^2 \LL_2 R_1^2 \LL_1 Z_+^2)(X)$ comes \emph{after} the $R_1^2$), and the columns of the diagram are not globular. Fortunately we can show that there is a pre-existing relation in $(R_1^2 \LL_1 Z_+^2)(X)$ which allows us to bypass the interchange law and still obtain the desired result.
\end{rmk}

\subsection{Generators for zigzags}

We have just shown that the canonical inclusion $X \to \zig_{+}^2(X)$ promotes every $1$-morphism of $X$ to a left adjoint. We will now prove that the (co)units of these adjunctions are enough to generate $\zig_{+}^2(X)$ under composition.

\begin{lmm} \label{lmm:vert-decomp}
	Every $2$-morphism of $\zig_{+}^2(X)$ can be written as a vertical composition of zigzags $\sigma_k \circ_1 \sigma_{k-1} \circ_1 \dotsb \circ_1 \sigma_1$, where 
	\begin{enumerate}[label=(\alph*)]
		\item each $\sigma_i$ is in $\Sq^2_{++}(X)_{1,1}$ or $\Sq^2_{-+}(X)_{1,1}$,
		\item $\partial_1^0 \sigma_1$ and $\partial_1^1 \sigma_k$ are identities, and
		\item the compositions $\circ_1$ are computed in $\LL_1 Z_+^2(X)$.
	\end{enumerate}
\end{lmm}

\begin{proof}
	This is a restatement of the formula for $\LL_1$ and $\LL_2$ in \Cref{fct:formula-segalification} and the placement of these functors in the definition $\zig_{+}^2(X) = (R_2^2 \LL_2 R_1^2 \LL_1 Z_+^2)(X)$. Namely, first we use $\LL_1$ to form the zigzags, then $R_1^2$ to ensure that each zigzag is globular, and then we use $\LL_2$ to compose these zigzags freely.
\end{proof}

\begin{prp} \label{prp:horiz-decomp}
	Each zigzag $\sigma_k \circ_1 \sigma_{k-1} \circ_1 \dotsb \circ_1 \sigma_1$ as above can be written, in $\zig_{+}^2(X)$, as a composition of $2$-morphisms of the form $\eta_f = H_{f^{\op_1}} \circ_1 H_{f}$, $\varepsilon_f = E_f \circ_1 E_{f^{\op_1}}$, $\id_2(f^\to)$, or $\id_2(f^\leftarrow)$.
\end{prp}

\begin{proof}
	We will use the same strategy as in the proof of \Cref{prp:adj-in-zig}: we will find some $\alpha \in (R_1^2 \LL_1 Z_+^2)(X)_{1,2}$ such that $\alpha_{01}$ and $\alpha_{12}$ are free horizontal compositions of things of the form $\eta_f$, $\varepsilon_f$, $\id_2(f^\to)$, or $\id_2(f^\leftarrow)$ and such that $\alpha_{02}$ is the zigzag. Without loss of generality we may assume that the $(1,1)$-morphisms $\sigma_i$ alternate between being in $\Sq^2_{++}(X)$ and $\Sq^2_{-+}(X)$, i.e. we compose everything that is already composable in the starting double categories. To simplify the proof we will assume that $k = 2j$ is even and that $\sigma_1 \in \Sq^2_{++}(X)_{1,1}$, but exactly the same argument will produce a proof for the other three cases.
	
	First write every $\sigma_{2i-1}$ and $\sigma_{2i}$, for $i = 1, \dotsc, j$, as commutative squares:
	\begin{equation*}
		\sigma_{2i-1} \simeq \,
		\begin{tikzcd}
			w_i \ar[r, "f_i"] \ar[d, "h_i"'] & x_i \ar[d, "g_i"] \\
			y_i \ar[r, "k_i"'] & z_i
		\end{tikzcd}
		\qquad
		\sigma_{2i} \simeq \,
		\begin{tikzcd}
			\hat{x}_i \ar[d, "\hat{g}_i"'] & \hat{w}_i \ar[l, "\hat{f}_i"'] \ar[d, "\hat{h}_i"] \\
			\hat{z}_i & \hat{y}_i \ar[l, "\hat{k}_i"]
		\end{tikzcd}
	\end{equation*}
	Since the pairs $(\sigma_{2i-1}, \sigma_{2i})$ and $(\sigma_{2i}, \sigma_{2i+1})$ are composable we must have $\hat{g}_i \simeq g_i$ and $\hat{h}_{i} \simeq h_{i+1}$ for all valid values of $i$. Now use \Cref{rmk:E-H-decomp} to write
	\begin{equation*}
		\sigma_{2i-1} \simeq (\id_2(k_i^\to) \circ_1 E_{h_i}) \circ_2 (H_{g_i} \circ_1 \id_2(f_i^\to)), \qquad \sigma_{2i} \simeq (E_{\hat{h}_i^{\op_1}} \circ_1 \id_2(\hat{k}_i^\leftarrow)) \circ_2 (\id_2(\hat{f}_i^\leftarrow) \circ_1 H_{\hat{g}_i^{\op_1}}).
	\end{equation*}
	In particular we have $\alpha_{2i-1} \in \Sq^{2}_{++}(X)_{12}$ and $\alpha_{2i} \in \Sq^{2}_{-+}(X)_{12}$ such that
	\begin{align*}
		(\alpha_{2i-1})_{01} & \simeq H_{g_i} \circ_1 \id_2(f_i^\to), & (\alpha_{2i})_{01} & \simeq \id_2(\hat{f}_i^\leftarrow) \circ_1 H_{\hat{g}_i^{\op_1}}, \\
		(\alpha_{2i-1})_{12} & \simeq \id_2(k_i^\to) \circ_1 E_{h_i}, & (\alpha_{2i})_{12} & \simeq E_{\hat{h}_i^{\op_1}} \circ_1 \id_2(\hat{k}_i^\leftarrow), \\
		(\alpha_{2i-1})_{02} & \simeq \sigma_{2i-1}, & (\alpha_{2i})_{02} & \simeq \sigma_{2i},
	\end{align*}
	and so we have $\alpha := \alpha_{2j} \circ_1 \alpha_{2j-1} \circ_1 \dotsb \circ_1 \alpha_{2} \circ_1 \alpha_1 \in (\LL_1 Z_+^2(X))_{1,2}$. By assumption the zigzag $\sigma_k \circ_1 \dotsb \circ_1 \sigma_1$ is globular and therefore $h_1$ and $\hat{h}_{j}$ are identities, which also implies that $\alpha$ is globular, i.e. $\alpha \in (R_1^2 \LL_1 Z_+^2(X))_{1,2}$. Finally, we must verify that ($\dagger$) $\alpha_{01}$ and $\alpha_{02}$ are composites of (co)units and identities and that ($\ddagger$) $\alpha_{02}$ is the zigzag we started with. Notice that for any $a,b \in \{0,1,2\}$ we have
	\begin{equation*}
		\alpha_{ab} \simeq (\alpha_{2j})_{ab} \circ_1 (\alpha_{2j-1})_{ab} \circ_1 \dotsb \circ_1 (\alpha_{2})_{ab} \circ_1 (\alpha_1)_{ab}
	\end{equation*}
	by the formula in \Cref{fct:formula-segalification}; this immediately implies ($\ddagger$). For ($\dagger$), note that
	\begin{equation*}
		(\alpha_{2j})_{01} \circ_1 (\alpha_{2j-1})_{01} \circ_1 \dotsb \circ_1 (\alpha_1)_{01} \simeq \id_2(\hat{f}_{j}^\leftarrow) \circ_1 \eta_{g_j} \circ_1 \id_2(f_j^\to) \dotsb \id_2(\hat{f}_{1}^\leftarrow) \circ_1 \eta_{g_1} \circ_1 \id_2(f_1^\to)
	\end{equation*}
	since $\hat{g}_i \simeq g_i$ and
	\begin{equation*}
		(\alpha_{2j})_{12} \circ_1 (\alpha_{2j-1})_{12} \circ_1 \dotsb \circ_1 (\alpha_1)_{12} \simeq E_{\hat{h}_j^{\op_1}} \circ_1 \id_2(\hat{k}_j^\leftarrow) \circ_1 \id_2(k_j^\to) \circ_1 \varepsilon_{h_j} \circ_1 \dotsb \circ_1 \varepsilon_{h_1} \circ_1 \id_2(\hat{k}_j^\leftarrow) \circ_1 \id_2(k_j^\to) \circ_1 E_{h_1}.
	\end{equation*}
	Together with the fact that $E_{h_1}$ and $E_{\hat{h}_j^{\op_1}}$ are identities (since $h_1$ and $\hat{h}_j$ are identities) these two equations prove ($\ddagger$), and so we're done.
\end{proof}

\Cref{lmm:vert-decomp} and \Cref{prp:horiz-decomp} allow us to conclude the following generation result:
\begin{crl} \label{crl:final-generation}
	 Let $A \hookrightarrow \zig^2(X)_1$ denote the space containing all the $2$-morphisms of the form
	\begin{enumerate}
		\item $\id(\id(x))$ for $x \in X_0$,
		\item $\id(f^\to)$ and $\id(f^\leftarrow)$ for all $f \in X_1$,
		\item $\eta_f$ and $\varepsilon_f$, as defined in \Cref{prp:horiz-decomp}, for all $f \in X_1$.
    \end{enumerate}
    Then $A$ generates $\zig^2(X)$ under composition.
\end{crl}

\begin{crl} \label{crl:factorization-through-inclusion}
	If $\sD$ is a sinister $2$-category and $F : X \to \sD$ is a map from a Segal space, the induced map $\tilde{F} : \zig_{+}^2(X) \to \zig_{+}^2(\sD)$ factors through the canonical inclusion $\sD \to \zig_{+}^2(\sD)$.
\end{crl}

\begin{proof}
	It's enough to show that the generating morphisms of $\zig_{+}^2(X)$ are sent to morphisms in $\sD$. This is obvious for objects. Morphisms coming from $X$ are sent to morphisms in $\sD$ by construction. If $f^\leftarrow$ is a morphism coming from $X^{\op_1}$ then $\tilde{F}(f^\leftarrow) \simeq (F(f))^{\leftarrow}$, and the latter is equivalent to any right adjoint $g$ of $F(f)$ in $\sD$ via the invertible $2$-morphism represented by the zigzag
	\begin{equation*}
		\begin{tikzcd}
			y \ar[d, equal] & x \ar[l, "F(f)"'] \ar[r, equal] \ar[d, "F(f)"'] & x \ar[d, equal] \ar[dl, "\eta", Rightarrow] \\
			y & y \ar[l, equal] \ar[r, "g"'] & x
		\end{tikzcd}
	\end{equation*}
	in $\zig_{+}^2(\sD)$. The same holds for the (co)units, since $\tilde{F}(E_f)$ and $\tilde{F}(H_f)$ are equivalent to the (co)units $\varepsilon$ and $\eta$ of an adjunction $F(f) \dashv g$ using \cite[Theorem 4.4.18]{RV2016}.
\end{proof}

\subsection{Universal property}

In this subsection we will prove our main result:
\begin{thm}\label{thm:univ-prop}
	Let $\sD$ be a sinister $(\infty,2)$-category. Then, for any Segal space $X$, the map
	\begin{equation*}
		i^\ast : \Map(\zig_{+}^2(X), \sD) \to \Map(X, \sD) \simeq \Map(X, \tau_1 \sD)
	\end{equation*}
	obtained by precomposing with the canonical inclusion $i : X \to \zig_{+}^2(X)$ is an equivalence of spaces.
\end{thm}

\begin{proof}
	To show that $i^\ast$ is a monomorphism, let $F : \zig_{+}^2(X) \to \sD$ be a functor and consider $E := i^\ast(F)$. By \Cref{crl:final-generation} the map
    \begin{equation*}
        \Map(\zig^2_+(X), \sD) \to \Map(M(A(1)^\el), \sD)
    \end{equation*}
    is a monomorphism; recall that $A$ contains all of $X_0$, $X_1$, $X_1^{\op_1}$, and the (co)units of the adjunctions $f^\to \dashv f^\leftarrow$. Therefore the inclusion $X \to \zig^2_+(X)$ factors through $M(A(1)^\el)$, and we can ask whether the map
    \begin{equation*}
        \Map(M(A(1)^\el), \sD) \to \Map(X, \sD) \simeq \Map(X, \tau_1 \sD)
    \end{equation*}
    is a monomorphism, which will imply that $i^\ast$ is a monomorphism. The fiber of this map over a functor $E : X \to \tau_1 \sD$ is seen to be equivalent to $\prod_{f \in X_1} \Adj_\sD(E(f))$, where, for any $g \in \sD_{1,0}$
	\begin{equation*}
		\Adj_\sD(g) \simeq \Map(\Adj, \sD) \times_{\sD_{1,0}} \{g\}
	\end{equation*}
	is the space of adjunction data that have $g$ as a left adjoint. But since $\sD$ is sinister each of these spaces is non-empty and, therefore, contractible by \cite[Theorem 4.4.18]{RV2016}. Hence the fiber over $E$ is contractible, as desired.
	
	For essential surjectivity, we will produce an explicit section of $i^\ast$ on path components. First note that there is a map $p : \zig_{+}^2(\tau_1 \sD) \to \sD$ such that its restriction to $\tau_1 \sD$ is the inclusion $j : \tau_1 \sD \to \sD$: indeed, $j$ induces a map $\zig_{+}^2(\tau_1 \sD) \to \zig_{+}^2(\sD)$ which factors as $\zig_{+}^2(\tau_1 \sD) \to \sD \to \zig_{+}^2(\sD)$ thanks to \Cref{crl:factorization-through-inclusion}, and the restriction to $\tau_1 \sD$ returns $j$ by construction. Now we have a map
	\begin{equation*}
		p_\ast \circ \zig_{+}^2 : \pi_0 \Map(X, \tau_1 \sD) \to \pi_0 \Map(\zig_{+}^2(X), \zig_{+}^2(\tau_1 \sD)) \to \pi_0 \Map(\zig_{+}^2(X), \sD)
	\end{equation*}
	such that $i^\ast \circ p_\ast \circ \zig_{+}^2 \simeq p_\ast \circ i^\ast \circ \zig_{+}^2 \simeq \id$, i.e. $p_\ast \circ \zig_{+}^2$ is a section of $i^\ast$. This concludes the proof.
\end{proof}

The following are immediate consequences of \Cref{thm:univ-prop}:

\begin{crl} \label{crl:univ-prop}
	Let $\sD$ be any $(\infty,2)$-category and let $\tau_1 \sD^{\mathrm{ladj}} \sub \tau_1 \sD$ denote the full sub-$\infty$-category containing the $1$-morphisms which are left adjoints in $\sD$. Then the inclusion $X \to \zig_{+}^2(X)$ induces an equivalence
	\begin{equation*}
		\Map(\zig_{+}^2(X), \sD) \simeq \Map(X, \tau_1 \sD^{\mathrm{ladj}}).
	\end{equation*}
\end{crl}

\begin{crl} \label{crl:adj}
	There is an equivalence $\zig_{+}^2([1]) \simeq \Adj$ between the zigzagification of the walking arrow and the free adjunction.
\end{crl}

The last corollary can be gleamed from the graphical representation of the $2$-morphisms of $\zig_{+}^2([1])$. Indeed, in this setting we have exactly two generating $2$-morphisms, $\eta$ and $\varepsilon$, the unit and counit for the adjunction $f^\to \dashv f^\leftarrow$. On top of the zigzags representing $\id_2(f^\to)$, $\id_2(f^\leftarrow)$, $\id_1(f)$, $\eta$ and $\varepsilon$ we can draw lines connecting the midpoints of the arrows in each commutative diagram, as in \Cref{fig:unit-counit}.
\begin{figure}[h!]
	\centering
	\begin{tikzcd}[column sep = tiny, row sep = tiny]
		\bullet \ar[rr] \ar[dd, equal] & \ar[dd, color = red, dash, start anchor = {[yshift = 1ex]}] & \bullet \ar[dd, equal] \\
		& {} & \\
		\bullet \ar[rr] & {} & \bullet
	\end{tikzcd}
	\qquad
	\begin{tikzcd}[column sep = tiny, row sep = tiny]
		\bullet \ar[dd, equal] & \ar[dd, color = red, dash, start anchor = {[yshift = 1ex]}] & \bullet \ar[ll] \ar[dd, equal] \\
		& {} & \\
		\bullet & {} & \bullet \ar[ll]
	\end{tikzcd}
	\qquad
	\begin{tikzcd}[column sep = tiny, row sep = tiny]
		\bullet \ar[rr, equal] \ar[dd] & & \bullet \ar[dd] \\
		\ar[rr, color = red, dash, start anchor = {[xshift = -0.5ex]}, end anchor = {[xshift = 0.5ex]}] & {} & {} \\
		\bullet \ar[rr, equal] & & \bullet
	\end{tikzcd}
	\qquad
	\begin{tikzcd}[column sep = tiny, row sep = tiny]
		\bullet \ar[dd, equal] \ar[rr, equal] & & \bullet \ar[dd] & & \bullet \ar[ll, equal] \ar[dd, equal] \\ 
		& & {} & & \\
		\bullet \ar[rr] & \ar[rr, red, controls={+(0.5,0.8) and +(-0.5,0.8)}, dash, shift right = 0.2] & \bullet & {} & \bullet \ar[ll]
	\end{tikzcd}
	\qquad
	\begin{tikzcd}[column sep = tiny, row sep = tiny]
		\bullet \ar[dd, equal] & \ar[rr, red, controls={+(0.5,-0.8) and +(-0.5,-0.8)}, dash, shift left = 1.4] & \bullet \ar[ll] \ar[rr] & {} & \bullet \ar[dd, equal] \\ 
		& & {} & & \\
		\bullet \ar[rr, equal] & & \bullet \ar[from=uu] & & \bullet \ar[ll, equal]
	\end{tikzcd}
	\caption{The identities, the unit and the counit as zigzags (in black) and as curves (in red).}
	\label{fig:unit-counit}
\end{figure}
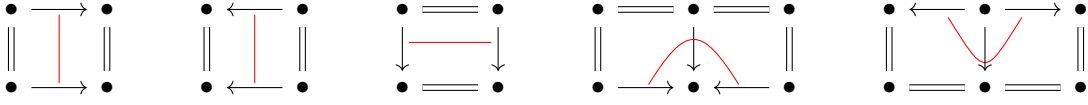

Using this graphical notation, one of the snake equations reads
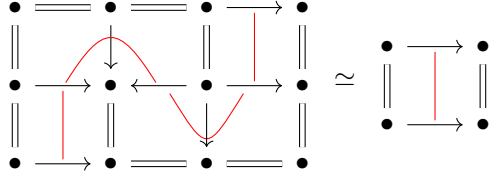
\begin{equation*}
	\begin{tikzcd}[column sep = tiny, row sep = tiny]
		\bullet \ar[dd, equal] \ar[rr, equal] & & \bullet \ar[dd] & & \bullet \ar[ll, equal] \ar[dd, equal] \ar[rr] & \ar[dd, color = red, dash, start anchor = {[yshift = 1ex]}] & \bullet \ar[dd, equal] \\ 
		& & {} & & & & \\
		\bullet \ar[rr] \ar[dd, equal] & \ar[rr, red, controls={+(0.5,0.8) and +(-0.5,0.8)}, dash, shift right = 0.2] \ar[dd, color = red, dash, start anchor = {[yshift = 1ex]}] & \bullet \ar[dd, equal] & {} \ar[rr, red, controls={+(0.5,-0.8) and +(-0.5,-0.8)}, dash, shift left = 1.4] & \bullet \ar[ll] \ar[rr] & {} & \bullet \ar[dd, equal] \\
		& & {} & & & & \\
		\bullet \ar[rr] & {} & \bullet \ar[rr, equal] & & \bullet \ar[rr, equal] \ar[from=uu] & & \bullet
	\end{tikzcd}
	\, \simeq \,
	\begin{tikzcd}[column sep = tiny, row sep = tiny]
		\bullet \ar[rr] \ar[dd, equal] & \ar[dd, color = red, dash, start anchor = {[yshift = 1ex]}] & \bullet \ar[dd, equal] \\
		& {} & \\
		\bullet \ar[rr] & {} & \bullet
	\end{tikzcd}
\end{equation*}
and the other is similar. After pasting these drawings together using zigzags we re-obtain the graphical calculus for $\Adj$ of \cite[Section 3.1]{RV2016} and of \cite[Section 6, Example 2]{DPP2003}.

\section{Higher-dimensional behavior of zigzagification} \label{sec:extensions}

We conclude the paper with a speculative section where we provide a few observations about the general case of $\zig_{+}^{n+1}(X)$ for $X \in \Glb^{n,n}(\Spaces)$ and hint at a possible universal property.

\subsection{Adjoints from zigzags of lax commutative $(n+1)$-cubes}

Recall briefly how we built right adjoints of $1$-morphisms $f \in X_1$ in $\zig_{+}^2(X)$: we first produced a companionship for $f^\to$ in $\Sq^2(X)$, which in particular gave us commutative squares $E_f$ and $H_f$, and then glued those commutative squares with their horizontal opposites to obtain (co)units $\varepsilon$ and $\eta$ for the adjunction $f^\to \dashv f^\leftarrow$. 

\begin{prp} \label{prp:all-adjoints}
	Every $k$-morphism $\alpha$ of $X$ has a right adjoint in $\zig_{+}^{n+1}(X)$.
\end{prp}

The proof of this statement is similar to that of \Cref{prp:adj-in-zig}. By induction, using \Cref{prp:truncation-zigzag}, the low-dimensional adjunctions ($k < n$) are taken care of in $\zig_{+}^{k+1}(\tau_{k} X)$, so we can assume $k = n$.

\begin{dfn} \label{dfn:companions}
	Let $Y \in \Seg^{n+1}(\Spaces)$. 
	\begin{itemize}
		\item If $n = 1$, a \emph{($1$-dimensional) companionship} in $Y$ is the data of $f \in X_{1,0}$, $g \in X_{0,1}$, and $H, E \in X_{1,1}$ satisfying $E \circ_1 H \simeq \id_1(g)$ and $E \circ_2 H \simeq \id_2(f)$.
		\item if $n \geq 2$, a \emph{($n$-dimensional) companionship} is a companionship in the double Segal space $Y_{1, \dotsc, 1, \bullet, \bullet}$.
	\end{itemize}
\end{dfn}

\begin{rmk} \label{rmk:1-d-companions}
	The content of \Cref{rmk:E-H-decomp} is that every $(1,0)$-morphism of $\Sq^2(X)$ extends to a $1$-dimensional companionship.
\end{rmk}

\begin{prp} \label{prp:mainthm2}
	If $X \in \GlbSeg^{n,n}(\Spaces)$ then every $(\undl{1^n}, 0)$-morphism of $\Sq^{n+1}(X)$ extends to an $n$-dimensional companionship. In particular, every $n$-morphism $\alpha$ of $X$ has a right adjoint $\alpha^{\leftarrow}$ in $\zig_{+}^{n+1}(X)$.
\end{prp}

\begin{proof}
	Using the fact that the Gray tensor product has a right adjoint (see \cite{Campion2023}), say $A \otimes - \dashv R_A$, we see that $\Sq^{n+1}(X)_{\undl{1^{n-1}}, \bullet, \bullet}$ is equivalent to $\Sq^2(\tau_2 R_{\square(n-1)}(X))$. By definition, $n$-dimensional companionships in the former are $1$-dimensional companionships in the latter, and so we're done by \Cref{rmk:1-d-companions}.
	
	An $n$-morphism $\alpha$ of $X$ induces $(\undl{1^n}, 0)$-morphisms $\alpha^\to$ of $\Sq^{n+1}(X)$ and $\alpha^{\leftarrow}$ of $\Sq^{n+1}(X)^{\op_n}$, together with a companionship between $\alpha^\to$ and the $(\undl{1^{n-1}}, 0, 1)$-morphism $\alpha^\downarrow$ in $\Sq^{n+1}(X)$. Note that $\alpha^\to$ and $\alpha^\leftarrow$ are globular up to height $n$ and $\alpha^\downarrow$ is globular up to height $n-1$, meaning that the companionship extends to $(R_{n-1}^{n+1} \LL_{n-1} \dotsb R_1^{n+1} \LL_1 Z_+^{n+1})(X)$. Now the same argument used in \Cref{prp:adj-in-zig} (with $R_1^2 \LL_1$ replaced by $R_n^{n+1} \LL_n$) shows that the the companionship data for $\alpha^\to$ gives two $(n+1)$-morphisms $\varepsilon_\alpha$ and $\eta_\alpha$ of $\zig_{+}^{n+1}(X)$ satisfying the snake equations.
\end{proof}

\begin{crl} \label{crl:ambidextrous}
	For $k \leq n-1$, every $k$-morphism $\alpha$ of $X$ has an ambidextrous adjoint in $\zig_{+}^{n+1}(X)$: the $n$-morphism $\alpha^\leftarrow$ satisfies $\alpha^\to \dashv \alpha^\leftarrow$ and $\alpha^\leftarrow \dashv \alpha^\to$.
\end{crl}

\begin{proof}
	The construction of $\zig_{+}^{n+1}(X)$ is invariant under the application of $(-)^{\op_k}$ for $k \leq n-1$, meaning that $\zig_{+}^{n+1}(X)^{\op_k} \simeq \zig_{+}^{n+1}(X)$. This is because $\Sq^{n+1}_{\undl{b},+}(X)^{\op_k} \simeq \Sq^{n+1}_{\undl{b'},+}(X)$ with $b_i' = b_i$ for $i \neq k$ and $b_k' = -b_k$. Under these equivalences, $\alpha^\to$ is sent to $\alpha^{\leftarrow}$ and viceversa. But in $\zig_{+}^{n+1}(X)^{\op_k}$ we have $\alpha^\leftarrow \dashv \alpha^\to$, and so the claim follows. 
\end{proof}

\begin{rmk}
	In fact, we believe that, for $k \leq n-1$, \emph{every} $k$-morphism of $\zig_{+}^{n+1}(X)$ has an ambidextrous adjoint. Such $k$-morphisms are given by taking $k$ zigzags of lax commutative $k$-cubes in $\tau_k X$ of signature $(\undl{b}, c) \in \Lambda^k$. Each cube has a counterpart obtained by considering the same cube but with signature $(\undl{b}, -c) \in \Lambda^k$. Rewriting the $k$ zigzags with the counterparts yields the desired adjoint, with (co)units obtained using the companionships of each cube.
\end{rmk}

\subsection{On generators}

The proof of \Cref{prp:horiz-decomp} depends on a certain decomposition result for $(1,1)$-morphisms of $\Sq^2(X)$, which is one part of the following universal property of $\Sq^2$: praphrasing the results of \cite{LR2025} (and deliberately omitting issues of completeness for simplicity's sake), $\Sq^2(X)$ is initial among those double Segal spaces which admit a map from $X$ that sends every $1$-morphism $f$ to a $(1,0)$-morphism with a companion.

It follows that if we wanted to prove a generators-and-relations result for $\zig_{+}^{n+1}(X)$ we could start by showing a similar universal property for $\Sq^{n+1}$. At the moment, however, it is unclear what such a property would say. In \Cref{dfn:companions} we defined an $n$-dimensional companionship for an $(n+1)$-uple Segal $Y$ to be a companionship in $Y_{1, \dotsc, 1, \bullet, \bullet}$, but we could reasonably form different double Segal spaces out of $Y$, such as $Y_{\bullet, \bullet, 1, \dotsc, 1}$ or $Y_{\bullet, 1, \dotsc, 1, \bullet}$, and ask for companionships there. Let's call them \emph{alternative companionships} to distinguish them from those of \Cref{dfn:companions}. It turns out that, in general, $\Sq^{n+1}(X)$ admits \emph{some} alternative companionships. For example, in $\Sq^3(X)$ we have that every $(1,0,1)$-morphism admits a companion $(0,1,1)$-morphism; not every $(1,1,0)$-morphism admits a companion $(0,1,1)$-morphism, but those $(1,1,0)$-morphisms that are globular up to height $1$ (i.e. that come from $2$-morphisms of $X$) do. In fact, implicit in the proof of \Cref{prp:all-adjoints} is the fact that every $k$-morphism of $X$, when considered as a $(\undl{1^k},\undl{0^{n+1-k}})$-morphism of $\Sq^{n+1}(X)$ via the canonical inclusion, admits a companion $(\undl{1^{k-1}},0,1,\undl{0^{n-k}})$-morphism.

The evidence from low dimensional cases and some wishful thinking leads us to formulate the following
\begin{cnj}
	Let $X \in \GlbSeg^{n,n}(X)$ and let $k \leq n+1$. The space of $(\undl{1^{k}},\undl{0^{n+1-k}})$-morphisms of $\Sq^{n+1}(X)$ is generated under composition by the $k$-morphisms of $X$ and the companionship (co)units for lower-dimensional morphisms of $X$.
\end{cnj}
The claim would be but one step in the proof of a possible universal property for $\Sq^{n+1}(X)$, stating something along the lines of ``$\Sq^{n+1}(X)$ is initial among those $(n+1)$-uple Segal spaces which admit a map from $X$ sending every $k$-morphism to a $(\undl{1^{k}},\undl{0^{n+1-k}})$-morphism with a (suitable) companion''. 

For our purposes, the claim gives us a grasp on the $k$-morphisms of $\zig_{+}^{n+1}(X)$ when $X \in \GlbSeg^{n+1}(X)$: since the companionship (co)units in each $\Sq^{n+1}_{\undl{b},+}(X)$ glue together to form (co)units for adjunctions of $k$-morphisms in $\zig_{+}^{n+1}(X)$, we expect the following:
\begin{cnj} \label{cnj:mainconj}
	Let $X \in \GlbSeg^{n,n}(X)$ and let $k \leq n+1$. Then the space of $k$-morphisms of $\zig_{+}^{n+1}(X)$ is generated under composition by the $k$-morphisms of $X$ and the adjunction (co)units for lower-dimensional morphisms of $X$.
\end{cnj}

According to the conjecture, maps $\zig_{+}^{n+1}(X) \to \sD$ to a sinister $(n+1)$-category are determined by their restriction to $X$. It remains a mystery if we can go the other way, i.e. if we can produce maps $\zig_{+}^{n+1}(X) \to \sD$ from maps $X \to \sD$, because we have little to no knowledge about the relations among adjunction (co)units in dimensions $n \geq 2$.

\subsection{Relation to cobordisms and future work}

As mentioned in \Cref{sec:intro}, our initial motivation was to produce a combinatorial construction for the cobordism higher category $\mathrm{Bord}_n^{\mathrm{fr}}$. It turns out that $\zig_{+}^{n}$ can be used to approximate a different but closely related higher category, the $\EE_k$-monoidal $(\infty,n)$-category of embedded oriented tangles. Here is a brief explanation of how this works. More details on this topic will be presented in future work.

Consider $F^k$, the free $\EE_k$-algebra in $\Spaces$ on one generator. Under the equivalence between $\EE_k$-algebras and $k$-fold monoids in cartesian monoidal $\infty$-categories (see, for example, \cite[Proposition 10.11]{Haugseng2018}), $F^k$ induces a globular $k$-uple Segal space $B^k F^k$ with a contractible space of $j$-morphisms for all $0 \leq j < k$. Then $\zig_{+}^{k+n}(B^k F^k)$ is a globular $(k+n)$-uple Segal space with a contractible space of $j$-morphisms for all $0 \leq j < k$, and thus it induces an $\EE_k$-monoidal globular $n$-uple Segal space $T^{k,n} := \Omega^k \zig_{+}^{k+n}(B^k F^k)$.

We claim that the $j$-morphisms of $T^{k,n}$ can be turned into the data of compact oriented $j$-dimensional manifolds with corners. First note that $F^k$ can be modeled by the space $\mathrm{Conf}(k)$ of embeddings of points in $(0,1)^k$, the open unit $k$-cube; the $\EE_k$-monoidal structure is given by pasting two cubes along a face and then rescaling. Paths in this space are isotopies of embeddings, paths between paths are isotopies between isotopies, and so on. In particular, any map $f : [0,1]^r \to \mathrm{Conf}(k)$ from the closed $r$-cube can be ``realized'' as an oriented manifold: first factor $f$ as a map $[0,1]^r \to \mathrm{Emb}(m, (0,1)^k)$ into the space of embeddings of $m$ points $\{p_1, \dotsc, p_m\}$ into $(0,1)^k$, and then define
\begin{equation*}
	M_f := \{(x, y) \sub [0,1]^r \times (0,1)^k \mid y = f(x)(p_i) \text{ for some } i = 1, \dotsc, m\} \sub [0,1]^r \times (0,1)^k \sub \RR^{r + k}
\end{equation*}
Note that $M_f$ is abstractly diffeomorphic to the disjoint union of $m$ copies of $[0,1]^r$, but as a manifold with corners it might be non-trivial.

Now consider a map $\alpha : \square^{k+j} \to B^k F^k \simeq B^k \mathrm{Conf}(k)$ of $(\infty,k+j)$-categories. Composing with the inclusion $w(k+j) \to \square^{k+j}$ gives us a specific $(k+j)$-morphism of $B^k \mathrm{Conf}(k)$. For $j = 0$ this corresponds to a point of $\mathrm{Conf}(k)$, for $j = 1$ this is a path in $\mathrm{Conf}(k)$, for $j = 2$ it's a path between paths, and so on. In general it corresponds to a certain map $a : [0,1]^j \to \mathrm{Conf}(k)$ and thus, via the realization, a compact oriented $j$-dimensional manifold with corners $M_a$. The boundary of $M_a$ is given by the realization of the faces of $\square^{k+j}$ along $\alpha$.

The discussion above can be summarized by saying that $(k+j)$-dimensional morphisms of $\Sq^{k+j}(B^k F^k)$ can be turned into compact oriented $j$-dimensional manifolds with corners embedded in $\RR^{k+j}$. Similarly, a $(k+j)$-dimensional morphism $\alpha$ of $\Sq^{k+j}_{\undl{b}}(B^k F^k)$ gives the same manifold $M_a$ as the corresponding morphism $\alpha'$ in $\Sq^{k+j}(B^k F^k)$ but we introduce a ``formal orientation'' depending on the parity of $\sum_{i} b_i$: if the latter is even then the orientation of $M_a$ matches that of $M_{a'}$ and if the latter is odd then the orientation of $M_a$ is the opposite one.

Ultimately, our hope is the following:
\begin{cnj} \label{cnj:secondcnj}
	The above association from maps $\square^r \to B^k F^k$ to manifolds yields a map $T^{k,n} = \Omega^k \zig_{+}^{k+n}(B^k F^k) \to \mathrm{Tang}_{k,n}^{\mathrm{or}}$ into the $\EE_k$-monoidal $(\infty,n)$-category of oriented tangles defined in \cite{AF2017}.
\end{cnj}
Analyzing this map might provide more insight into the \emph{tangle hypothesis}, an $\EE_k$-monoidal analogue of the cobordism hypothesis, formulated in \cite{AF2017} and proven conditionally on a different conjecture about factorization homology. As explained in the introduction, taking colimits would yield a map $\colim_k T^{k,n} \to \Bord_{n}^{\orient}$ that might help us understand the oriented version of the cobordism hypothesis.

\end{document}